\documentclass[a4paper,british,american,final]{article}
\pdfoutput=1
\usepackage[sc]{mathpazo}

\usepackage[T1]{fontenc}
\usepackage[latin9]{inputenc}
\usepackage{color}
\usepackage{babel}
\usepackage{units}
\usepackage{mathtools}
\usepackage{amsmath}
\usepackage{amsthm}
\usepackage{amssymb}
\usepackage{setspace}
\usepackage{esint}
\usepackage{xargs}[2008/03/08]
\usepackage[unicode=true,
 breaklinks=false,pdfborder={0 0 0},backref=false,colorlinks=true]
 {hyperref}
\hypersetup{pdftitle={The Number of Nodal Components of Arithmetic Random Waves},
 pdfauthor={Yoni Rozenshein},
 linkcolor=NavyBlue, citecolor=DarkGreen, urlcolor=FireBrick}
\usepackage[numbered=true,open=false,depth=2]{bookmark}

\makeatletter

\pdfpageheight\paperheight
\pdfpagewidth\paperwidth

\usepackage{enumitem}		
\numberwithin{equation}{section}
\theoremstyle{plain}
\newtheorem{thm}{\protect\theoremname}[section]
  \theoremstyle{definition}
  \newtheorem{defn}[thm]{\protect\definitionname}
  \theoremstyle{remark}
  \newtheorem{rem}[thm]{\protect\remarkname}
  \theoremstyle{plain}
  \newtheorem{prop}[thm]{\protect\propositionname}
  \theoremstyle{plain}
  \newtheorem{lem}[thm]{\protect\lemmaname}
\newenvironment{lyxcode}
{\par\begin{list}{}{
\setlength{\rightmargin}{\leftmargin}
\setlength{\listparindent}{0pt}
\raggedright
\setlength{\itemsep}{0pt}
\setlength{\parsep}{0pt}
\normalfont\ttfamily}%
 \item[]}
{\end{list}}
  \theoremstyle{definition}
  \newtheorem{condition}[thm]{\protect\conditionname}
  \theoremstyle{remark}
  \newtheorem*{rem*}{\protect\remarkname}
  \theoremstyle{plain}
  \newtheorem{assumption}[thm]{\protect\assumptionname}

\@ifundefined{date}{}{\date{}}
\usepackage[all,error]{onlyamsmath}

\usepackage[margin=1in]{geometry}

\usepackage[font=small,justification=centering]{caption}

\usepackage{bbm}

\usepackage[svgnames]{xcolor}

\usepackage{microtype}

\usepackage{tikz}

\makeatother

  \addto\captionsamerican{\renewcommand{\assumptionname}{Assumption}}
  \addto\captionsamerican{\renewcommand{\conditionname}{Condition}}
  \addto\captionsamerican{\renewcommand{\definitionname}{Definition}}
  \addto\captionsamerican{\renewcommand{\lemmaname}{Lemma}}
  \addto\captionsamerican{\renewcommand{\propositionname}{Proposition}}
  \addto\captionsamerican{\renewcommand{\remarkname}{Remark}}
  \addto\captionsamerican{\renewcommand{\theoremname}{Theorem}}
  \addto\captionsbritish{\renewcommand{\assumptionname}{Assumption}}
  \addto\captionsbritish{\renewcommand{\conditionname}{Condition}}
  \addto\captionsbritish{\renewcommand{\definitionname}{Definition}}
  \addto\captionsbritish{\renewcommand{\lemmaname}{Lemma}}
  \addto\captionsbritish{\renewcommand{\propositionname}{Proposition}}
  \addto\captionsbritish{\renewcommand{\remarkname}{Remark}}
  \addto\captionsbritish{\renewcommand{\theoremname}{Theorem}}
  \providecommand{\assumptionname}{Assumption}
  \providecommand{\conditionname}{Condition}
  \providecommand{\definitionname}{Definition}
  \providecommand{\lemmaname}{Lemma}
  \providecommand{\propositionname}{Proposition}
  \providecommand{\remarkname}{Remark}
\providecommand{\theoremname}{Theorem}

\begin{document}

\newcommand{\bigslant}[2]{{\left.\raisebox{.2em}{$#1$}\middle/\raisebox{-.2em}{$#2$}\right.}}

\global\long\def\quotient#1#2{\bigslant{#1}{#2}}

\global\long\def\Span{\operatorname{Span}}

\global\long\def\vol{\operatorname{vol}}

\global\long\def\id{\operatorname{id}}
\global\long\def\Mat{\operatorname{Mat}}

\global\long\def\re{\operatorname{Re}}
\global\long\def\im{\operatorname{Im}}

\global\long\def\rank{\operatorname{rank}}
\global\long\def\img{\operatorname{im}}

\global\long\def\diam{\operatorname{diam}}
\global\long\def\dist{\operatorname{dist}}

\global\long\def\Length{\operatorname{Length}}
\global\long\def\Count{\operatorname{Count}}
\global\long\def\Clopen{\operatorname{Clopen}}

\global\long\def\var{\operatorname{var}}
\global\long\def\cov{\operatorname{cov}}
\global\long\def\median{\operatorname{Median}}

\global\long\def\td{\mathbb{T}^{d}}
\global\long\def\Td{\mathcal{T}^{d}}
\global\long\def\zd{\mathbb{Z}^{d}}
\global\long\def\rd{\mathbb{R}^{d}}
\global\long\def\rk{\mathbb{R}^{k}}
\global\long\def\rn{\mathbb{R}^{n}}
\global\long\def\rm{\mathbb{R}^{m}}

\global\long\def\R{\mathbb{R}}
\global\long\def\c{\mathbb{C}}
\global\long\def\z{\mathbb{Z}}
\global\long\def\q{\mathbb{Q}}
\global\long\def\n{\mathbb{N}}
\global\long\def\t{\mathbb{T}}

\global\long\def\p{\mathbb{P}}
\global\long\def\E{\mathbb{E}}

\global\long\def\Z{\mathcal{Z}}
\global\long\def\D{\mathcal{D}}
\global\long\def\N{\mathcal{N}}
\global\long\def\F{\mathbf{F}}

\global\long\def\e{\mathrm{e}}
\global\long\def\i{\mathrm{i}}

\global\long\def\Empty{\varnothing}

\global\long\def\epsilon{\varepsilon}
\global\long\def\tilde#1{\widetilde{#1}}
\global\long\def\ll{\lesssim}
\global\long\def\gg{\gtrsim}

\global\long\def\sums{\Lambda_{L}}
\global\long\def\sumsplus{\Lambda_{L}^{+}}
\global\long\def\sumsnorm{\widetilde{\Lambda}_{L}}

\global\long\def\hilbert{\mathcal{H}_{L}}

\global\long\def\sphere{\mathcal{S}^{d-1}}
\global\long\def\nsphere{\mathcal{S}^{n-1}}

\global\long\def\trigs{\mathcal{P}_{D}}
\global\long\def\laplacedom{\mathcal{D}_{\Delta}}

\global\long\def\matrices{\Mat_{n}\left(\R\right)}

\newcommandx\Int[4][usedefault, addprefix=\global, 1=, 2=]{\int_{#1}^{#2}#3\,\mathrm{d}#4}

\global\long\def\indic{\mathbbm{1}}

\title{\noindent \textbf{\Huge{}The Number of Nodal Components of Arithmetic
Random Waves}}

\author{Yoni Rozenshein\thanks{School of Mathematical Sciences, Tel Aviv University, Tel Aviv, Israel.
Email: \protect\href{mailto:yoni.vl@gmail.com}{yoni.vl@gmail.com}}}
\maketitle
\begin{abstract}
We study the number of nodal components (connected components of the
set of zeroes) of functions in the ensemble of arithmetic random waves,
that is, random eigenfunctions of the Laplacian on the flat $d$-dimensional
torus $\td$ ($d\ge2$). Let $f_{L}$ be a random solution to $\Delta f+4\pi^{2}L^{2}f=0$
on $\td$, where $L^{2}$ is a sum of $d$ squares of integers, and
let $N_{L}$ be the random number of nodal components of $f_{L}$.
By recent results of Nazarov and Sodin, $\E\left\{ N_{L}/L^{d}\right\} $
tends to a limit $\nu>0$, depending only on $d$, as $L\to\infty$
subject to a number-theoretic condition - the equidistribution on
the unit sphere of the normalized lattice points on the sphere of
radius $L$. This condition is guaranteed when $d\ge5$, but imposes
restrictions on the sequence of $L$ values when $2\le d\le4$. We
prove the exponential concentration of the random variables $N_{L}/L^{d}$
around their medians and means (unconditionally) and around their
limiting mean $\nu$ (under the condition that it exists).
\end{abstract}

\section{Introduction and presentation of the results}

\subsection{Toral eigenfunctions and arithmetic random waves}

Let $\hilbert\subset L^{2}\left(\td\right)$ be the real Hilbert space
of Laplacian eigenfunctions on the torus, i.e.\ functions $f\colon\td\to\R$
satisfying the partial differential equation:
\[
\Delta f+4\pi^{2}L^{2}f=0.
\]
We consider $d\ge2$ to be a fixed dimension (all ``constants''
mentioned below may depend on $d$); $L$ may vary. It is known that
the spectrum of eigenvalues is discrete; eigenfunctions exist whenever
$L^{2}$ can be expressed as a sum of $d$ squares of integers, and
then,
\[
\hilbert=\Span\left\{ \cos\left(2\pi\lambda\cdot x\right),\sin\left(2\pi\lambda\cdot x\right):\lambda\in\sums\right\} ,
\]
where $\sums=\left\{ \lambda\in\zd:\left|\lambda\right|=L\right\} $.
Each $\lambda$ generates the same functions as $-\lambda$, so $\dim\hilbert=\#\sums$.

For any $f\colon\td\to\R$, we denote by $Z\left(f\right)$ its \emph{nodal
set }(the subset of $\td$ where $f$ vanishes), and by $N\left(f\right)$
the number of its \emph{nodal components }(the connected components
of the nodal set). In this paper, we address the question: \emph{What
is the typical behavior of $N\left(f\right)$ for $f\in\hilbert$,
with fixed $d$ and large $L$?}

Typically (when $f$ and $\nabla f$ do not vanish simultaneously),
the number of nodal components almost equals the number of \emph{nodal
domains} (the connected components of $\td\setminus Z\left(f\right)$)
- they cannot differ by more than $d-1$. Thus, Courant's nodal domain
theorem gives a general upper bound $N\left(f\right)\ll L^{d}$, with
an explicit constant. Unfortunately, a general, non-trivial lower
bound cannot be obtained, as there are classical counterexamples with
arbitrarily large $L$ and only two nodal domains, originally shown
in \cite{Stern-thesis} (see also \cite{Berard-Helffer-Dirichlet-eigenfunctions},
\cite{Bruning-Fajman-On-the-nodal-count-for-flat-tori}).

It is expected, however, that such eigenfunctions with high eigenvalue
but few nodal components are outliers, and $N\left(f\right)$ is in
the order of magnitude of $L^{d}$ for ``most'' $f\in\hilbert$.
To study the typical case, we refer to a probabilistic model that
was introduced and investigated in \cite{Oravecz-Rudnick-Wigman-the-Leray-measure-of-nodal-sets,Rudnick-Wigman-the-volume-of-nodal-sets}.
Consider the random function $f_{L}\colon\td\to\R$:
\begin{equation}
f_{L}\left(x\right)\coloneqq\sqrt{\frac{2}{\dim\hilbert}}\sum_{\lambda\in\sumsplus}\left(a_{\lambda}\cos\left(2\pi\lambda\cdot x\right)+b_{\lambda}\sin\left(2\pi\lambda\cdot x\right)\right),\label{eq:fL}
\end{equation}
where the set $\sumsplus=\quotient{\sums}{\pm}$ is half of the set
$\sums$ (representatives of the equivalence $\lambda\sim\pm\lambda$),
and $a_{\lambda},b_{\lambda}$ are random variables, i.i.d.\ $\N\left(0,1\right)$.
The sequence of functions $\left\{ f_{L}\right\} $ is called the
\emph{ensemble of arithmetic random waves}. The random function $f_{L}$
may be viewed as a random element of the finite-dimensional space
$\hilbert$, or as a centered, stationary Gaussian process, normalized
such that $\E\left\{ \left|f_{L}\left(x\right)\right|^{2}\right\} =1$,
with covariance kernel:
\begin{equation}
K_{L}\left(x,y\right)\coloneqq\E\left\{ f_{L}\left(x\right)f_{L}\left(y\right)\right\} =\frac{1}{\dim\hilbert}\sum_{\lambda\in\sums}\cos\left(2\pi\lambda\cdot\left(x-y\right)\right).\label{eq:KL}
\end{equation}
Note that due to rotation invariance, the definition of $f_{L}$ does
not depend on the choice of basis for $\hilbert$.

Under this probabilistic model, the number of nodal components $N_{L}\coloneqq N\left(f_{L}\right)$
becomes a random variable (we discuss its measurability in detail
in Section \ref{sec:measurability}) and the question of its behavior
may be formulated in terms of expected value and concentration as
$L\to\infty$.

\subsection{Asymptotic law for \texorpdfstring{$\protect\E\left\{ N_{L}\right\} $}{E\{NL\}}}

Nazarov and Sodin \cite{Nazarov-Sodin-asymptotic-laws} (see also
lecture notes \cite{Sodin-SPB-Lecture-Notes}) proved, in a much more
general setting of ensembles of Gaussian functions on Riemannian manifolds,
an asymptotic law for the expected value of $N\left(f\right)$. Our
first theorem, Theorem \ref{thm:asymptotic-law}, is simply a formulation
of the Nazarov-Sodin theorem, applied to our case. There is one obstacle:
The theorem requires the existence of a limiting spectral measure
satisfying certain properties. In our case, this limiting spectral
measure does not necessarily exist, and it depends on the following
number-theoretic equidistribution condition:
\begin{defn}
\label{def:admissible-sequence}A sequence of values of $L$ that
tends to infinity, with $L^{2}$ always a sum of $d$ squares, is
called an \emph{admissible sequence of $L$ values} if the integer
points on the sphere of radius $L$, when projected onto the unit
sphere, become equidistributed as $L\to\infty$. In other words,
\begin{equation}
\frac{1}{\dim\hilbert}\sum_{\lambda\in\sums}\delta_{\lambda/L}\Rightarrow\sigma_{d-1},\label{eq:weak-convergence-of-measures}
\end{equation}
where ``$\Rightarrow$'' indicates weak-{*} convergence of measures,
and $\sigma_{d-1}$ is the uniform measure on the unit sphere, with
the normalization $\sigma_{d-1}\left(\sphere\right)=1$.
\end{defn}
This equidistribution condition depends on the dimension $d$. When
$d\ge5$, any sequence of $L$ values is admissible, whereas in the
low dimensions $2\le d\le4$, some conditions must be satisfied. For
more on the subject, see Appendix \ref{sec:appendix-equidistribution}.
\begin{thm}
\label{thm:asymptotic-law}There is a constant $\nu>0$ such that:
\[
\begin{array}{cc}
\E\left\{ N_{L}\right\} \sim\nu L^{d} & \text{ as }L\to\infty\text{ through any admissible sequence}.\end{array}
\]

\end{thm}

\subsection{Main result: Exponential concentration of \texorpdfstring{$N_{L}$}{NL}}

Our main result is that $N_{L}$ concentrates around its median, mean,
and limiting mean, exponentially in $\dim\hilbert$, i.e.\ the number
of independent random variables. This is similar to a previous result
by Nazarov and Sodin in the case of random spherical harmonics \cite{Nazarov-Sodin-spherical-harmonics}.
\begin{thm}
\label{thm:exponential-concentration}Let $\varepsilon>0$. There
exist constants $C\left(\varepsilon\right),c\left(\varepsilon\right)>0$
such that:
\begin{enumerate}
\item For any $L$:
\[
\p\left\{ \left|\frac{N_{L}}{L^{d}}-\median\left\{ \frac{N_{L}}{L^{d}}\right\} \right|>\varepsilon\right\} \le C\left(\epsilon\right)\e^{-c\left(\epsilon\right)\dim\hilbert}.
\]

\item For any $L$:
\[
\p\left\{ \left|\frac{N_{L}}{L^{d}}-\E\left\{ \frac{N_{L}}{L^{d}}\right\} \right|>\varepsilon\right\} \le C\left(\epsilon\right)\e^{-c\left(\epsilon\right)\dim\hilbert}.
\]

\item If $\mathbb{E}\left\{ N_{L}L^{-d}\right\} $ tends to a limit $\nu$
through some sequence of $L$ values, then for any large enough $L$
in this sequence:
\[
\p\left\{ \left|\frac{N_{L}}{L^{d}}-\nu\right|>\varepsilon\right\} \le C\left(\varepsilon\right)\e^{-c\left(\varepsilon\right)\dim\hilbert}.
\]

\end{enumerate}
In all cases, our proof yields $c\left(\varepsilon\right)\gg\epsilon^{\left(d+2\right)^{2}-1}$.\end{thm}
\begin{rem}
\label{rem:understanding-the-main-result}In the low dimensions $2\le d\le4$,
it is possible to have sequences of $L$ values for which $\dim\hilbert$
stays bounded as $L$ grows. However, if $\dim\hilbert$ is bounded,
then all parts of Theorem \ref{thm:exponential-concentration} are
completely trivial, and say nothing. This is unlike the case of random
spherical harmonics discussed in \cite{Nazarov-Sodin-spherical-harmonics},
in which the dimension is a simple ascending function of the eigenvalue.

The second and third parts of Theorem \ref{thm:exponential-concentration}
are straightforward consequences of the first part. This is proven
in Subsection \ref{sub:concentration-implies-concentration}. When
$d\ge3$, the third part of the theorem only makes sense when $\nu$
is the same $\nu$ from Theorem \ref{thm:asymptotic-law}. This is
because under the assumption that $\dim\hilbert$ is not bounded from
below (without which, the theorem says nothing anyway), the limit
(\ref{eq:weak-convergence-of-measures}) holds and the sequence is
admissible. However, when $d=2$, we could have a limiting measure
other than $\sigma_{d-1}$ in (\ref{eq:weak-convergence-of-measures}),
so value of $\nu$ in the third part of Theorem \ref{thm:exponential-concentration}
truly depends on the chosen sequence of $L$ values. See also Appendix
\ref{sec:appendix-equidistribution} and \cite{Kurlberg-Wigman-Non-universality}.
\end{rem}

\subsection{Outline of the paper}

In Section \ref{sec:measurability}, we prove the Borel measurability
of the random variable $N_{L}$ - the number of nodal components.
We deduce it from a more general result - the measurability of the
number of nodal components of more general random functions (Proposition
\ref{prop:measurability-generalization}), which may be of independent
interest.

In Section \ref{sec:proof-of-usability-of-Nazarov-Sodin-theorem},
we show that Theorem \ref{thm:asymptotic-law} follows directly from
the Nazarov-Sodin theorem.

In Section \ref{sec:trigonometric-polynomials-and-nodal-sets}, we
treat trigonometric polynomials in general, and give algebraic proofs
to bounds on the sum of diameters of their connected components.

In Section \ref{sec:proof-of-main-theorem}, we present a proof of
Theorem \ref{thm:exponential-concentration}.

In Appendix \ref{sec:appendix-equidistribution}, we provide background
and quote the known results on the problem of equidistribution of
lattice points on spheres.

In Appendix \ref{sec:appendix-additional-proofs}, we provide proofs
for some of the claims used in the paper.

\subsection{Notation}

We reserve the letters $C$ and $c$ for positive constants (usually
upper and lower bounds, respectively) which may vary from line to
line; all constants may depend on the dimension $d$. When $A,B$
are positive quantities, we denote by $A\ll B$, $A\gg B$ and $A\simeq B$
that $A\le CB$, $A\ge cB$ and $cB\le A\le CB$, respectively.

The notation $K_{+\delta}$ indicates the set of all points of distance
at most $\delta$ from the compact set $K$, which may be a set in
$\td$ or $\rd$ with Euclidean distance or $L^{2}\left(\td\right)$
with the norm-induced distance.

\subsection{Acknowledgments}

This work was written following and based on my master's thesis in
Tel Aviv University. I thank my advisor, Mikhail Sodin, for providing
me with the opportunity to work on such a diverse project; his invaluable
guidance and patience has made this work possible. I also thank Zeév
Rudnick, Lior Bary-Soroker, Eugenii Shustin, and my fellow students,
for many discussions and helpful advice. I also thank the anonymous
referee, whose fruitful comments made it possible to improve the presentation
of the paper. This work, and the author's master degree studies, were
partly supported by grant N\textsuperscript{\underline{o}} 166/11
of the Israel Science Foundation of the Israel Academy of Sciences
and Humanities.

\section{\label{sec:measurability}Measurability of \texorpdfstring{$N_{L}$}{NL}}

The random process $f_{L}$ is, formally, a function $f_{L}\colon\td\times\Omega\to\R$,
where $\Omega$ is the Gaussian probability space. In this point of
view, $N_{L}$ is the function on $\Omega$ given by $\omega\mapsto N\left(f_{L}\left(\cdot,\omega\right)\right)$.
\begin{prop}
\label{prop:NL-is-measurable}$N_{L}$ is a random variable. In other
words, the mapping $N_{L}\colon\Omega\to\n\cup\left\{ 0,\infty\right\} $
is measurable.
\end{prop}
Naturally, the countable set $\n\cup\left\{ 0,\infty\right\} $ is
equipped with the discrete (power set) $\sigma$-algebra.

There is an ``analytic'' way to prove that $N_{L}$ is at least
Lebesgue measurable. Skipping some details, this proof is as follows:
By Bulinskaya's lemma, the event that $f_{L}$ does not have a ``stable
nodal set'' is a subset of $\Omega$ of Lebesgue measure zero (see
Definition \ref{def:stable-nodal-set} and Proposition \ref{prop:a.s.-stable-nodal-set}
ahead). The complement of this event is open because, since $\dim\hilbert$
is finite, the norm $\left\Vert g\right\Vert $ of any small perturbation
$g\in\hilbert$ bounds both $\max\left|g\right|$ and $\max\left|\nabla g\right|$.
$N_{L}$ is locally constant in this open set (see Proposition \ref{prop:smooth-perturbation}
ahead). Thus, any event of the form $\left\{ N_{L}=n\right\} $ is
a union of an open set and a subset of an event of Lebesgue measure
zero.

The above proof has two weaknesses: The first is that it makes assumptions
on the function space. To use Bulinskaya's lemma, we require the smoothness
of the functions, the finite dimension of $\td$, and a condition
on the probability density of the random process, and to continue
the proof, we need the finite dimension of $\hilbert$. The second
weakness is that the proof only yields Lebesgue measurability.

In this section, we present a strong generalization of Proposition
\ref{prop:NL-is-measurable} and prove it using only the most basic
definitions in topology and measure theory:
\begin{prop}
\label{prop:measurability-generalization}Let $X$ be a compact metric
space, let $\Omega$ be a (not necessarily complete) probability space,
and let $f\colon X\times\Omega\to\R$ be a random real-valued function
on $X$ that is a.s.\ continuous. Then the number of nodal components
of $f$ is a random variable, i.e.\ a measurable mapping $\Omega\to\n\cup\left\{ 0,\infty\right\} $.
\end{prop}
This immediately implies Proposition \ref{prop:NL-is-measurable}
(where $X=\td$) with Borel measurability. To prove Proposition \ref{prop:measurability-generalization},
observe that the number of nodal components of $f$ is the composition
of three maps: 
\[
\Omega\xrightarrow{\quad f\quad}C\left(X\right)\xrightarrow{\quad Z\quad}{\bf F}\left(X\right)\xrightarrow{\quad\Count\quad}\mathbb{N}\cup\left\{ 0,\infty\right\} 
\]

\begin{itemize}
\item The first map $\omega\mapsto f\left(\cdot,\omega\right)$ sends almost
every $\omega\in\Omega$ to a corresponding function in $C\left(X\right)$.
\item The second map $f\mapsto Z\left(f\right)$ sends a continuous real
function to its (closed) zero set.
\item The third map $F\mapsto\Count\left(F\right)$ counts the number of
connected components of a given closed set $F\subset X$.
\end{itemize}
In the remainder of this section, we show that all three maps are
measurable with respect to the following (standard) $\sigma$-algebras
on $C\left(X\right)$ and $\F\left(X\right)$:
\begin{itemize}
\item The standard $\sigma$-algebra on $C\left(X\right)$ is generated
by the \emph{point-evaluation maps} $\left\{ f\mapsto f\left(x\right)\right\} _{x\in X}$;
that is, it is generated by the family of sets $\left\{ f\in C\left(X\right):f\left(x\right)\in B\right\} $,
where $x$ varies over all points in $X$ and $B$ varies over all
Borel subsets of $\R$.
\item The standard $\sigma$-algebra on $\F\left(X\right)$ is given by
the following equivalent definitions (see also chapters 2.4 and 3.3
of \cite{Srivastava-book-A-course-on-Borel-sets}):

\begin{enumerate}
\item The $\sigma$-algebra $\mathcal{F}_{1}$ generated by the family of
sets $\left\{ F\in{\bf F}\left(X\right):F\cap U\ne\Empty\right\} $,
where $U$ varies over all open subsets of $X$.
\item The $\sigma$-algebra $\mathcal{F}_{2}$ generated by the family of
sets $\left\{ F\in\mathbf{F}\left(X\right):F\subset U\right\} $,
where $U$ varies over all open subsets of $X$.
\item The $\sigma$-algebra $\mathcal{F}_{3}$ generated by the family of
sets $\left\{ F\in\mathbf{F}\left(X\right):F\cap K\ne\Empty\right\} $,
where $K$ varies over all compact subsets of $X$.\end{enumerate}
\begin{proof}[Proof that $\mathcal{F}_{1}=\mathcal{F}_{2}=\mathcal{F}_{3}$]
Any subset of $X$ is compact if and only if it is closed, so the
sets $\left\{ F\in\mathbf{F}\left(X\right):F\subset U\right\} $ and
$\left\{ F\in\mathbf{F}\left(X\right):F\cap K\ne\Empty\right\} $
are complementary when taking $K=X\setminus U$, and we get $\mathcal{F}_{2}=\mathcal{F}_{3}$.

Any open subset of $X$ is a countable union of compact subsets, so
any set $\left\{ F\in{\bf F}\left(X\right):F\cap U\ne\Empty\right\} $
is a countable union of sets $\left\{ F\in\mathbf{F}\left(X\right):F\cap K\ne\Empty\right\} $,
and we have $\mathcal{F}_{1}\subset\mathcal{F}_{3}$. Similarly, any
compact subset is a countable intersection of open subsets, so any
set $\left\{ F\in\mathbf{F}\left(X\right):F\cap K\ne\Empty\right\} $
is a countable intersection of sets $\left\{ F\in{\bf F}\left(X\right):F\cap U\ne\Empty\right\} $,
and we have $\mathcal{F}_{3}\subset\mathcal{F}_{1}$.
\end{proof}
\end{itemize}

It is evident that the first map $\Omega\to C\left(X\right)$ is measurable:
By the definition of the $\sigma$-algebra on $C\left(X\right)$,
this is equivalent to the map $\Omega\to\R$ given by $\omega\mapsto f\left(x,\omega\right)$
being measurable for any $x\in X$, which is precisely the definition
of $f$ being a random function.

The second map is measurable by the following:
\begin{prop}
Let $X$ be a compact metric space. The map $Z\colon C\left(X\right)\to{\bf F}\left(X\right)$
given by $f\mapsto f^{-1}\left(\left\{ 0\right\} \right)$ is measurable.\end{prop}
\begin{proof}
We will show that for any compact $K\subset X$, the set $\left\{ F\in{\bf F}\left(X\right):F\cap K\ne\Empty\right\} $
has a measurable preimage under $Z$ - that is, that the set $\left\{ f\in C\left(X\right):\exists x\in K:f\left(x\right)=0\right\} $
is measurable in $C\left(X\right)$. Let $A\subset K$ be countable
and dense in $K$. By a standard continuity argument, we have:
\begin{equation}
\left\{ f\in C\left(X\right):\exists x\in K:f\left(x\right)=0\right\} =\bigcap_{\varepsilon>0}\bigcup_{x\in A}\left\{ f\in C\left(X\right):\left|f\left(x\right)\right|<\varepsilon\right\} .\label{eq:all-f-with-zero-in-K}
\end{equation}
The sets on the right hand side are generating sets in the $\sigma$-algebra
of $C\left(X\right)$, so the set on the left hand side is measurable,
proving the proposition.
\end{proof}
The measurability of the third map - the component counting function
- is a little trickier. We first show a couple of lemmas that will
help translate the number of components to a property of covers by
open sets, which is more easily expressed by generators of the $\sigma$-algebra
on $\F\left(X\right)$.

For any topological space $Y$, we denote by $\Clopen\left(Y\right)$
the Boolean algebra of clopen (that is, closed and open) subsets of
$Y$. Note that $Y$ is connected if and only if $\Clopen\left(Y\right)=\left\{ \Empty,Y\right\} $,
and that if $A\in\Clopen\left(Y\right)$ then $\Clopen\left(A\right)\subset\Clopen\left(Y\right)$.
Recall that any clopen set is a union of connected components, and
that connected components are always closed.
\begin{lem}
\label{lem:counting-components-in-the-clopen-algebra}Let $Y$ be
a topological space and let $N$ be a positive integer. The following
are equivalent:
\begin{enumerate}
\item $Y$ has strictly fewer than $N$ connected components.
\item For any $Y_{1},\ldots,Y_{N}\in\Clopen\left(Y\right)$, if they are
pairwise disjoint then one of them is empty.
\end{enumerate}
\end{lem}
\begin{proof}
The second condition follows from the first simply by the pigeonhole
principle, since clopen sets are unions of connected components. Conversely,
assume the second condition holds. Suppose that $\Clopen\left(Y\right)$
is an infinite collection of sets. Let $Y_{1}\in\Clopen\left(Y\right)$.
Any clopen set $A$ is the union of $A\cap Y_{1}$ and $A\setminus Y_{1}$,
both clopen, so either $\Clopen\left(Y_{1}\right)$ is infinite or
$\Clopen\left(Y\setminus Y_{1}\right)$ is infinite. We may thus iteratively
construct a sequence $Y_{1},Y_{2},\ldots\in\Clopen\left(Y\right)$
of nonempty pairwise disjoint sets, contradicting the assumed condition.
Therefore, $\Clopen\left(Y\right)$ must be finite, and being a finite
Boolean algebra, it is generated by a finite number of atoms - clopen
sets with no clopen subsets. These atoms are precisely the connected
components, and applying the assumed condition on the atoms, there
must be fewer than $N$ of them.
\end{proof}
For the following, recall that any compact metric space $X$ admits
a countable collection $\mathcal{U}$ of open sets which separates
closed sets: For any closed, pairwise disjoint $F_{1},\ldots,F_{n}\subset X$,
there are pairwise disjoint $U_{1},\ldots,U_{n}\in\mathcal{U}$ with
$F_{i}\subset U_{i}$. We skip the proof of this, which follows easily
from the fact that any compact metric space is second-countable and
normal.
\begin{lem}
\label{lem:N-components-separated-by-open-sets}Let $X$ be a compact
metric space, let $\mathcal{U}$ be a collection of open sets in $X$
which separates closed sets, let $F\subset X$ be closed and let $N$
be a positive integer. The following are equivalent:
\begin{enumerate}
\item $F$ has strictly fewer than $N$ connected components.
\item For any pairwise disjoint $U_{1},\ldots,U_{N}\in\mathcal{U}$ such
that $F\subset\bigcup_{i=1}^{N}U_{i}$, there is a proper subcover
(i.e.\ one of the sets $U_{i}$ does not intersect $F$).
\end{enumerate}
\end{lem}
\begin{proof}
Suppose the first condition holds. Let $U_{1},\ldots,U_{N}\in\mathcal{U}$
be pairwise disjoint sets that cover $F$. Each $F\cap U_{i}$ is
relatively open in $F$, and its $F$-complement $F\cap\left(\bigcup_{j\ne i}U_{j}\right)$
is also relatively open in $F$. Therefore, $F\cap U_{i}\in\Clopen\left(F\right)$.
By Lemma \ref{lem:counting-components-in-the-clopen-algebra}, one
of the $F\cap U_{i}$ must be empty.

Conversely, suppose the second condition holds. We will show that
for any $F_{1},\ldots,F_{N}\in\Clopen\left(F\right)$ that are pairwise
disjoint, one of them must be empty, and then we are done by Lemma
\ref{lem:counting-components-in-the-clopen-algebra}.

Without loss of generality, we assume $F=\bigcup_{i=1}^{N}F_{i}$
(otherwise, we replace $F_{1}$ with $F\setminus\bigcup_{i=2}^{N}F_{i}$).
Since $F_{1},\ldots,F_{N}$ are relatively closed in $F$, they are
closed in $X$. Let $U_{1},\ldots,U_{N}\in\mathcal{U}$ be pairwise
disjoint such that $F_{i}\subset U_{i}$. By the hypothesis of the
second condition, one of the $U_{i}$ does not intersect $F$, meaning
one of the $F_{i}$ is empty.\end{proof}
\begin{prop}
Let $X$ be a compact metric space. The map $\Count\colon\F\left(X\right)\to\mathbb{N}\cup\left\{ 0,\infty\right\} $,
that counts the number of connected components in the given closed
set, is measurable.\end{prop}
\begin{proof}
The sets $\left\{ 0\right\} ,\left\{ 0,1\right\} ,\left\{ 0,1,2\right\} ,\ldots$
generate the $\sigma$-algebra on $\mathbb{N}\cup\left\{ 0,\infty\right\} $,
so it is enough to show that given a positive integer $N$, the following
set is measurable:
\[
\F_{<N}\coloneqq\left\{ F\in{\bf F}\left(X\right):F\text{ has strictly fewer than }N\text{ connected components}\right\} .
\]

Let $\mathcal{U}$ be a countable collection of open sets separating
closed sets. By Lemma \ref{lem:N-components-separated-by-open-sets},
$\F_{<N}$ may be written as the set of all $F\in\F\left(X\right)$
such that for any pairwise disjoint $U_{1},\dots,U_{N}\in\mathcal{U}$,
if $F\subset\bigcup_{i=1}^{N}U_{i}$ then one of the $U_{i}$ does
not intersect $F$:
\begin{align}
{\bf F}_{<N} & =\bigcap_{\substack{U_{1},\ldots,U_{N}\in\mathcal{U}\\
\text{pairwise disjoint}
}
}\left\{ F\in\mathbf{F}\left(X\right):\left(\left(F\subset\bigcup_{i=1}^{N}U_{i}\right)\implies\left(\exists i:F\cap U_{i}=\Empty\right)\right)\right\} \nonumber \\
 & =\bigcap_{\substack{U_{1},\ldots,U_{N}\in\mathcal{U}\\
\text{pairwise disjoint}
}
}\left\{ F\in\mathbf{F}\left(X\right):\left(F\not\subset\bigcup_{i=1}^{N}U_{i}\right)\text{ or }\left(\exists i:F\cap U_{i}=\Empty\right)\right\} \nonumber \\
 & =\bigcap_{\substack{U_{1},\ldots,U_{N}\in\mathcal{U}\\
\text{pairwise disjoint}
}
}\left(\left\{ F\in\mathbf{F}\left(X\right):F\not\subset\bigcup_{i=1}^{N}U_{i}\right\} \cup\bigcup_{i=1}^{N}\left\{ F\in\mathbf{F}\left(X\right):F\cap U_{i}=\Empty\right\} \right).\label{eq:final-form-of-F<N}
\end{align}
Subsets of $\F\left(X\right)$ of the form $\left\{ F\in\mathbf{F}\left(X\right):F\not\subset U\right\} $
and $\left\{ F\in\mathbf{F}\left(X\right):F\cap U=\Empty\right\} $,
where $U\subset X$ is open, are basic measurable sets in the $\sigma$-algebra
on $\F\left(X\right)$. Thus, the expression under the (countable)
intersection in (\ref{eq:final-form-of-F<N}) evaluates to a measurable
set, and we get that $\mathbf{F}_{<N}$ is measurable.
\end{proof}

\section{\label{sec:proof-of-usability-of-Nazarov-Sodin-theorem}Proof of
Theorem \ref{thm:asymptotic-law}}

The Nazarov-Sodin theorem \cite{Nazarov-Sodin-asymptotic-laws} (see
also lecture notes \cite{Sodin-SPB-Lecture-Notes}) gives an asymptotic
law for the expected number of connected components of Gaussian functions
under very general conditions. In this section, we show that Theorem
\ref{thm:asymptotic-law} is a specialization of the Nazarov-Sodin
theorem for our case.

We begin by computing the objects $K_{x,L}\left(u,v\right)$ and $C_{x,L}\left(u\right)$,
as they are defined in \cite[Sections 2.2-2.3]{Sodin-SPB-Lecture-Notes}.
In our case, $K_{L}$ is given by (\ref{eq:KL}). First, the scaled
covariance kernel $K_{x,L}\left(u,v\right)$:
\begin{equation}
K_{x,L}\left(u,v\right)=K_{L}\left(x+\frac{u}{L},x+\frac{v}{L}\right)=K_{L}\left(\frac{u-v}{L}\right)=\frac{1}{\dim\hilbert}\sum_{\lambda\in\sums}\cos\left(2\pi\frac{\lambda}{L}\cdot\left(u-v\right)\right).\label{eq:KxL}
\end{equation}
Note that this expression for $K_{x,L}\left(u,v\right)$ satisfies
the definition for $C^{3,3}$-smoothness of the ensemble (see \cite[Definition 2]{Sodin-SPB-Lecture-Notes},
where it is called ``separate $C^{3}$-smoothness''). To see this,
it is enough to show that the partial derivative $\partial_{u}^{i}\partial_{v}^{j}K_{x,L}\left(u,v\right)$
with $0\le i,j\le3$ remains uniformly bounded. Note that this partial
derivative is given by an expression similar to (\ref{eq:KxL}), where
the function in the sum is either $\pm\cos$ or $\pm\sin$ (depending
on $i+j$), and the addend corresponding to $\lambda$ is multiplied
by $\pm2\pi\lambda_{m}L^{-1}\in\left[-2\pi,2\pi\right]$ every time
a derivative is taken with respect to $u_{m}$ or $v_{m}$; therefore,
$\left|\partial_{u}^{i}\partial_{v}^{j}K_{x,L}\left(u,v\right)\right|\le\left(2\pi\right)^{6}$.

Second, the scaled covariance matrix $C_{x,L}\left(u\right)$. By
the following computation, we have that $C_{x,L}\left(u\right)$ is
simply a constant multiple of the identity matrix:
\begin{align*}
\left(C_{x,L}\left(u\right)\right)_{ij} & =\left.\partial_{u_{i}}\partial_{v_{j}}K_{x,L}\left(u,v\right)\right|_{v=u}=\left.\frac{1}{\dim\hilbert}\sum_{\lambda\in\sums}\left(2\pi\right)^{2}\frac{\lambda_{i}\lambda_{j}}{L^{2}}\cos\left(2\pi\lambda\cdot\frac{u-v}{L}\right)\right|_{v=u}\\
 & =\frac{1}{\dim\hilbert}\frac{\left(2\pi\right)^{2}}{L^{2}}\sum_{\lambda\in\sums}\lambda_{i}\lambda_{j}=\begin{cases}
\hfill0\hfill & \text{if }i\ne j\\
\hfill{\displaystyle \frac{4\pi^{2}}{d}}\hfill & \text{if }i=j
\end{cases}
\end{align*}
In the last step we have used an orthogonality relation that can be
seen easily by observing symmetries within the set $\sums$. Thus,
it clearly satisfies the definition for non-degeneracy of the ensemble
(see \cite[Definition 3]{Sodin-SPB-Lecture-Notes}).

Finally, we introduce our target limiting spectral measure for the
process: $\sigma_{d-1}$, the normalized Lebesgue measure on the sphere
$\sphere$, with $\sigma_{d-1}\left(\sphere\right)=1$. The target
translation-invariant local limiting covariance kernel $k$ is thus
the Fourier (cosine) transform of $\sigma_{d-1}$:
\[
k\left(x\right)\coloneqq\Int[\sphere]{\cos\left(2\pi x\cdot\zeta\right)}{\sigma_{d-1}\left(\zeta\right)}.
\]
Under the assumption that $L\to\infty$ through an admissible sequence
of $L$ values (Definition \ref{def:admissible-sequence}), we have
$K_{x,L}\left(u,v\right)\to k\left(x\right)$ pointwise in $x\in\rd$.
Pointwise convergence implies compact convergence in $\rd$ by a standard
application of the Arzelà-Ascoli theorem, and we get translation-invariant
local limits as in \cite[Definition 1]{Sodin-SPB-Lecture-Notes}.

Thus, Theorem \ref{thm:asymptotic-law} follows from \cite[Theorem 4]{Sodin-SPB-Lecture-Notes},
where the positivity of the constant $\nu$ for the spectral measure
$\sigma_{d-1}$ follows from condition $\left(\rho4\right)$ in \cite[Theorem 1]{Sodin-SPB-Lecture-Notes}.

\section{\label{sec:trigonometric-polynomials-and-nodal-sets}Trigonometric
polynomials and their nodal sets}

\subsection{Stability of nodal sets under small perturbations}

We begin by introducing some notation and definitions for the discussion
of the topological stability of a function's nodal set under small
perturbations. Let $f\colon\td\to\R$ be any continuous function.
We define $Z\left(f\right),\Z\left(f\right)$ and $N\left(f\right)$
by:
\begin{align*}
Z\left(f\right) & \coloneqq\left\{ x\in\td:f\left(x\right)=0\right\} =f^{-1}\left(\left\{ 0\right\} \right)\\
\Z\left(f\right) & \coloneqq\left\{ \text{connected components of }Z\left(f\right)\right\} \\
N\left(f\right) & \coloneqq\text{\# connected components of }Z\left(f\right)=\#\Z\left(f\right)
\end{align*}
$Z\left(f\right)$ is called the \emph{nodal set} of $f$, and its
connected components (which comprise $\Z\left(f\right)$) are called
the \emph{nodal components} of $f$. The connected components of $\td\setminus Z\left(f\right)$
are called the \emph{nodal domains} of $f$.
\begin{defn}
\label{def:stable-nodal-set}We say that a $C^{1}$-smooth function
$f\colon\td\to\R$ \emph{has a stable nodal set} if $\nabla f\left(x\right)\ne0$
for all $x\in Z\left(f\right)$.\end{defn}
\begin{rem}
\label{rem:about-nodal-stability}By compactness, the condition that
$f$ has a stable nodal set is equivalent to the existence of $\alpha,\beta>0$
such that for any $x\in\td$, $\left|f\left(x\right)\right|>\alpha$
or $\left|\nabla f\left(x\right)\right|>\beta$, and $\alpha,\beta$
may be chosen under a constraint $\alpha/\beta<\delta$ for any arbitrary
$\delta>0$. Note that if $f$ has a stable nodal set then $Z\left(f\right)$
is a $\left(d-1\right)$-dimensional smooth compact submanifold of
$\td$ having finitely many connected components.
\end{rem}
By the following two propositions, stable nodal sets are indeed stable
under small perturbations. Proposition \ref{prop:perturbation-in-U}
discusses ``local'' stability in an open subset of the torus, and
Proposition \ref{prop:smooth-perturbation} is a ``global'' version
(cf.\ \cite[Corollary 4.3]{Nazarov-Sodin-spherical-harmonics}).
These propositions may be proven in a standard way, by studying the
flow of the vector field $\left|\nabla f\right|^{-2}\nabla f$, and
for completeness, we provide their proofs in Appendix \ref{sec:appendix-additional-proofs}.
\begin{prop}
\label{prop:perturbation-in-U}Let $\alpha,\beta>0$ and let $U$
be an open subset of $\td$. Let $f\colon U\to\R$ be a smooth function
such that $\left|f\left(x\right)\right|>\alpha$ or $\left|\nabla f\left(x\right)\right|>\beta$
for any $x\in U$.

Let $g\colon U\to\R$ be a continuous function such that $\left|g\left(x\right)\right|<\alpha$
for any $x\in U$ (this is the ``small perturbation'').

Then for each connected component $\Gamma$ of $\left\{ x\in U:f\left(x\right)=0\right\} $
that satisfies $\Gamma_{+\alpha/\beta}\subset U$ , there is a connected
component $\tilde{\Gamma}\subset\Gamma_{+\alpha/\beta}$ of $\left\{ x\in U:f\left(x\right)+g\left(x\right)=0\right\} $.
Furthermore, the mapping $\Gamma\mapsto\tilde{\Gamma}$ is injective
(that is, different components $\Gamma$ generate different components
$\tilde{\Gamma}$).
\end{prop}

\begin{prop}
\label{prop:smooth-perturbation}Let $\alpha,\beta>0$ and let $f\colon\td\to\R$
be a smooth function such that $\left|f\left(x\right)\right|>\alpha$
or $\left|\nabla f\left(x\right)\right|>\beta$ for any $x\in\td$.

Let $g\colon\td\to\R$ be a smooth function such that $\left|g\left(x\right)\right|<\alpha/2$
and $\left|\nabla g\left(x\right)\right|<\beta/2$ for any $x\in\td$.

Then there is a bijection $\Z\left(f\right)\to\Z\left(f+g\right)$
mapping each $\Gamma\in\Z\left(f\right)$ to a corresponding $\tilde{\Gamma}\in\Z\left(f+g\right)$,
which satisfies:
\[
\diam\Gamma\le\frac{2\alpha}{\beta}+\diam\tilde{\Gamma}.
\]
\end{prop}
\begin{lyxcode}

\end{lyxcode}
The following proposition allows us to change our focus from the number
of nodal components to the number of nodal domains and vice versa,
by showing that their difference is very small. It is proven in a
standard way using singular homology theory. A proof is presented
in Appendix \ref{sec:appendix-additional-proofs}.
\begin{prop}
\label{prop:nodal-components-similar-to-nodal-domains}Let $f\colon\td\to\R$
be smooth with a stable nodal set.
\begin{enumerate}
\item If $f$ has $k$ nodal components and $r$ nodal domains, then $r-1\le k\le r+d-1$.
\item If $f$ has $k'$ nodal components and $r'$ nodal domains lying completely
inside some open ball of radius less than $\frac{1}{2}$, then $k'\le r'$.
\end{enumerate}
\end{prop}

\subsection{The number and sum of diameters of nodal components of trigonometric
polynomials}

We denote by $\trigs$ the linear space of trigonometric polynomials
on $\td$ of degree at most $D$:
\[
\trigs\coloneqq\Span\left\{ \cos\left(2\pi\lambda\cdot x\right),\sin\left(2\pi\lambda\cdot x\right):\lambda\in\zd,\left\Vert \lambda\right\Vert _{1}\le D\right\} .
\]

\begin{prop}
\label{prop:trig-polynomials-nodal-set-result}If $T\in\trigs$ has
a stable nodal set then $N\left(T\right)\ll D^{d}$ and $\sum_{\Gamma\in\Z\left(T\right)}\diam\Gamma\ll D^{d-1}$.
\end{prop}
The first result, the bound on the number of nodal components of $T$,
is a trigonometric version of a classical bound on the sum of the
Betti numbers of the nodal hypersurface of a polynomial due to Oleinik
and Petrovsky, Milnor, and Thom. We obtain this by result by algebraizing
(that is, writing the trigonometric polynomials as algebraic ones)
and applying elimination theory and Bézout's theorem to count the
number of critical points. This result is obtained along the way of
proving the second result, the bound on the sum of diameters, which
is shown using simple integral-geometric tools: The diameter of $\Gamma\in\Z\left(T\right)$
is comparable to its average width, which may be computed by measuring
the set of hypersurfaces that intersect it (a Crofton-type formula).

In \cite{Nazarov-Sodin-spherical-harmonics}, a different approach
is taken to bound the sum of diameters of nodal components in dimension
$d=2$ - it is bounded by a well-known estimate on the total length
of the nodal set. However, although this estimate may be generalized
to higher dimensions (where length is replaced by hypersurface volume),
when $d>2$ it fails to bound the sum of diameters; a nodal component
might be a long, thin ``noodle'' having large diameter and small
hypersurface volume.

It is likely that the stability condition in Proposition \ref{prop:trig-polynomials-nodal-set-result}
may be lifted, but we assume it as it makes the proof a little simpler,
and for $f_{L}$, this condition is almost surely satisfied (see Proposition
\ref{prop:a.s.-stable-nodal-set} in the next section).

\paragraph*{Algebraic background.}

Given polynomials $P_{1},\ldots,P_{n}\colon\c^{k}\to\c$, denote by
$Z\left(P_{1},\ldots,P_{n}\right)$ their common zero set. Subsets
of $\c^{k}$ of this form are called \emph{algebraic}, and their complements
\emph{coalgebraic}. The family of algebraic sets (thus also coalgebraic
sets) is closed under finite unions and finite intersections, and
any coalgebraic set is either empty or dense in $\c^{k}$ (since a
nonzero polynomial cannot vanish in an open set).

Any homogeneous polynomial of positive degree $P\colon\c^{n+1}\to\c$
has a trivial zero at the origin; other zeroes are called \emph{nontrivial
zeroes}. Note that if $P\left(z_{0,}\ldots,z_{n}\right)=0$, then
$P\left(\lambda z_{0},\ldots,\lambda z_{n}\right)=0$ for any $\lambda\in\c$,
so nontrivial zeroes extend at least to their complex spans, called
\emph{solution rays}.

In what follows, we consider polynomials in which some (possibly all)
of the coefficients are \emph{indeterminate}; that is, they are parameters
which may be assigned complex values, and on which we may impose conditions.
When considering polynomials with indeterminate coefficients, they
have a \emph{formal degree}, that is, the degree of the polynomial
with all the coefficients written explicitly; the actual degree may
become lower if some coefficients become zero after assigning values.

We recall two classical theorems. See, for instance, \cite[Sections 80 and 83]{van-der-Waerden-book-Modern-Algebra-vol-II}.
\begin{itemize}
\item \textbf{The fundamental theorem of elimination theory:} Given a system
of homogeneous polynomials of positive formal degree, the existence
of a nontrivial common zero is an algebraic condition on the coefficients.
\item \textbf{Bézout's theorem:} Given $n$ homogeneous polynomials $P_{1},\ldots,P_{n}\colon\c^{n+1}\to\c$,
if they have finitely many common solution rays, then the number of
common solution rays is at most $\prod_{i=1}^{n}\deg P_{i}$.
\end{itemize}
It is easy see that Bézout's\textbf{ }theorem continues to hold under
the weakened hypothesis that the polynomials have at most countably
many common solution rays; we will later name this result ``Bézout's
theorem'' as well.

\paragraph*{Coalgebraic condition for finiteness.}

We begin by presenting a coalgebraic condition for the finiteness
of the set of solutions to a system of polynomials which we will use
later in a few settings.

First, we define the \emph{homogenization }of a polynomial $P\colon\c^{n}\to\c$
of formal degree $D$ as the homogeneous polynomial $\tilde P\left(z_{0},\ldots,z_{n}\right)\coloneqq z_{0}^{D}P\left(\frac{z_{1}}{z_{0}},\ldots,\frac{z_{n}}{z_{0}}\right)$.
Note that plugging in $z_{0}=1$ yields the original polynomial, so
any zero $\left(\zeta_{1},\ldots,\zeta_{n}\right)$ of $P$ yields
a solution ray of $\tilde P$: $\Span_{\c}\left\{ \left(1,\zeta_{1},\ldots,\zeta_{n}\right)\right\} $.
However, $\tilde P$ may have more solution rays not obtained this
way - solution rays where $z_{0}=0$.

Second, we recall that the \emph{Jacobian} of $n$ polynomials $P_{1},\ldots,P_{n}\colon\c^{n}\to\c$
is the polynomial given by $\det\left(\frac{\partial P_{i}}{\partial z_{j}}\right)_{1\le i,j\le n}$.
In case that the Jacobian is nonzero wherever $P_{1},\ldots,P_{n}$
are zero, we get by the implicit function theorem that there are at
most countably many common zeroes of $P_{1},\ldots,P_{n}$.
\begin{condition}[Condition for finiteness]
\label{cond:finiteness}Polynomials $P_{1},\ldots,P_{n}\colon\c^{n}\to\c$
are said to satisfy the \emph{condition for finiteness} if both of
the following hold:
\begin{enumerate}
\item The homogenizations $\tilde P_{1},\ldots,\tilde P_{n}$ have no common
nontrivial zero with $z_{0}=0$.
\item At any common zero of $P_{1},\ldots,P_{n}$, their Jacobian is nonzero.
\end{enumerate}
\end{condition}
\begin{lem}
\label{lem:cond-finiteness-results}Condition \ref{cond:finiteness}
is a coalgebraic condition on the coefficients of the polynomials
$P_{1},\ldots,P_{n}$, and whenever it holds, the number of common
zeroes of $P_{1},\ldots,P_{n}$ is finite and bounded by $\prod_{i=1}^{n}\deg P_{i}$.\end{lem}
\begin{proof}
Let $P_{n+1}$ be the Jacobian of $P_{1},\ldots,P_{n}$. Replacing
the second part of the condition with ``$\tilde{P_{1}},\ldots,\tilde{P_{n}},\tilde{P_{n+1}}$
have no nontrivial common zero'' yields exactly the same condition,
since by the first part of the condition, there cannot exist such
a solution ray with $z_{0}=0$, and solution rays with $z_{0}\ne0$
are in correspondence with zeroes of $P_{1},\ldots,P_{n},P_{n+1}$.
By the fundamental theorem of elimination theory, this is a coalgebraic
condition on the coefficients of $P_{1},\ldots,P_{n}$ (and also $P_{n+1}$,
but the coefficients of $P_{n+1}$ are themselves polynomials of the
coefficients of $P_{1},\ldots,P_{n}$). Bézout's theorem then gives
the required bound for the number of common solution rays of $\tilde{P_{1}},\ldots,\tilde{P_{n}}$
with $z_{0}\ne0$, which are in correspondence with the zeroes of
$P_{1},\ldots,P_{n}$.
\end{proof}

\paragraph*{Algebraization of trigonometric polynomials.}

Laplacian eigenfunctions on the torus are trigonometric polynomials,
so it would be fruitful to consider trigonometric polynomials in general
in order to analyze them. Recall that trigonometric polynomials are
linear combinations of trigonometric monomials - cosines or sines
of $2\pi\lambda\cdot x$, where $\lambda\in\zd$. The degree of a
trigonometric monomial is $\left\Vert \lambda\right\Vert _{1}=\sum_{i=1}^{d}\left|\lambda_{i}\right|$,
and the degree of a trigonometric polynomial is the maximal degree
among its monomials. The polynomial is said to be homogeneous when
its monomials all have the same degree.

We would like to use algebraic tools, such as Bézout's theorem, to
analyze trigonometric polynomials. Therefore, we \emph{algebraize}
them - convert them to simple algebraic polynomials in a different
space.

In the following definition, we denote by $x_{1},\ldots,x_{d}$ the
coordinates of $\td$ and by $c_{1},s_{1},\ldots,c_{d},s_{d}$ the
coordinates of $\R^{2d}$.
\begin{defn}[algebraization of trigonometric polynomials]
$ $
\begin{enumerate}
\item Given a trigonometric polynomial $T\colon\td\to\R$, its \emph{algebraization}
is the polynomial $P\colon\R^{2d}\to\R$ such that $\deg P=\deg T$
and
\[
T\left(x_{1},\ldots,x_{d}\right)=P\left(\cos\left(2\pi x_{1}\right),\sin\left(2\pi x_{1}\right),\ldots,\cos\left(2\pi x_{d}\right),\sin\left(2\pi x_{d}\right)\right).
\]

\item Let $T_{1},\ldots,T_{r}$ be a system of $r$ trigonometric polynomials
$\td\to\R$. The \emph{algebraization }of this system is a system
of $r+d$ polynomials $\R^{2d}\to\R$, the first $r$ being the algebraizations
of $T_{1},\ldots,T_{r}$ and the last $d$ being the polynomials $c_{i}^{2}+s_{i}^{2}-1$
for $i=1,\ldots,d$.
\end{enumerate}
\end{defn}
\begin{rem*}
Given a trigonometric polynomial $T$, it is easy to see that its
algebraization $P$ exists and is unique. Moreover, the coefficients
of $P$ are linear combinations of the coefficients of $T$, and if
$T$ is homogeneous then $P$ is homogeneous. Regarding systems of
polynomials, it is easy to see that the assignments $c_{i}=\cos\left(2\pi x_{i}\right)$
and $s_{i}=\sin\left(2\pi x_{i}\right)$ for $i=1,\dots,d$ yield
a bijective correspondence between the set of zeroes of a system of
trigonometric polynomials and the set of (real) zeroes of the algebraized
system.
\end{rem*}

\paragraph*{Coalgebraic conditions for regularity of trigonometric polynomials.}

Let $T$ be a real trigonometric polynomial in $d$ variables $x_{1},\ldots,x_{d}$.
We present two regularity conditions, viewed as conditions on the
coefficients of $T$.
\begin{condition}[Condition for nodal regularity]
\label{cond:nodal-regularity}The homogenization of the algebraization
of the system $T,\frac{\partial T}{\partial x_{1}},\ldots,\frac{\partial T}{\partial x_{d}}$
has no common nontrivial complex zeroes.
\end{condition}
Condition \ref{cond:nodal-regularity} is coalgebraic by the fundamental
theorem of elimination theory, and it implies that $T$ has a stable
nodal set (recall Definition \ref{def:stable-nodal-set}).
\begin{condition}[Condition for critical set regularity]
\label{cond:critical-regularity}The algebraization of the system
$\frac{\partial T}{\partial x_{1}},\ldots,\frac{\partial T}{\partial x_{d}}$
satisfies the \emph{condition for finiteness} (Condition \ref{cond:finiteness}).
\end{condition}
By Lemma \ref{lem:cond-finiteness-results}, Condition \ref{cond:critical-regularity}
is coalgebraic and it implies that $\left\{ \nabla T=0\right\} $
is a finite set.
\begin{condition}[Condition for full regularity]
\label{cond:regular}Suppose $d>1$. \emph{$T$} is said to be \emph{fully
regular} if:
\begin{enumerate}
\item $T$ satisfies both the condition for nodal regularity and the condition
for critical set regularity.
\item For all $j=1,\ldots,d$, the restriction of $T$ to the hyperplane
$\left\{ x_{j}=0\right\} $, viewed as a trigonometric polynomial
in the remaining $d-1$ variables, satisfies both the condition for
nodal regularity and the condition for critical set regularity.
\end{enumerate}
\end{condition}
It is clear that Condition \ref{cond:regular} is coalgebraic in the
coefficients of $T$. We will later need our trigonometric polynomial
$T$ to have regular restrictions to almost any hyperplane of the
form $\left\{ x_{j}=t\right\} $, not just $\left\{ x_{j}=0\right\} $.
The following lemma shows that the full regularity condition is enough.
\begin{lem}
\label{lem:improved-regularity}Suppose $T$ is fully regular. Let
$1\le j\le d$. For all but finitely many values of $t\in\t$, the
restriction of $T$ to the hyperplane $\left\{ x_{j}=t\right\} $,
viewed as a trigonometric polynomial in the remaining $d-1$ variables,
satisfies both the condition for nodal regularity and the condition
for critical set regularity.\end{lem}
\begin{proof}
For any $t\in\t$, let $T_{t}$ be the restriction of $T$ to the
hyperplane $H\coloneqq\left\{ x_{j}=t\right\} $.

Put $\kappa\coloneqq\cos\left(2\pi t\right)$ and $\sigma\coloneqq\sin\left(2\pi t\right)$.
In what follows, $\kappa$ and $\sigma$ are treated as indeterminate
parameters, like $t$. The gradient $\nabla T_{t}$ comprises $d-1$
trigonometric polynomials on $\t^{d-1}$ (since the coordinate $x_{j}$
is omitted in the restriction to $H$), so the algebraization of the
system $T_{t},\nabla T_{t}$ is a system of $2d-1$ polynomials in
$2d-2$ variables ($c_{1},s_{1},\ldots,c_{d},s_{d}$ without $c_{j},s_{j}$)
whose coefficients are polynomial expressions in the coefficients
of $T$ (which are \emph{not }indeterminate) and the parameters $\kappa$
and $\sigma$.

Condition \ref{cond:nodal-regularity} and Condition \ref{cond:critical-regularity}
are coalgebraic; therefore, there exists a system of polynomials $P_{1}\left(\kappa,\sigma\right),P_{2}\left(\kappa,\sigma\right),\ldots$
that vanishes whenever $\kappa$ and $\sigma$ are assigned values
for which the conditions are not satisfied. At least one of these
polynomials (w.l.o.g.\ $P_{1}$) is not identically zero, because
both conditions are satisfied for $t=0$ (i.e.\ for $\kappa=1$ and
$\sigma=0$).

Consider the function $Q\left(t\right)\coloneqq P_{1}\left(\cos\left(2\pi t\right),\sin\left(2\pi t\right)\right)$.
Whenever $Q\left(t\right)\ne0$, $t$ is a value for which $T_{t}$
satisfies both conditions. We also know that $Q\left(0\right)\ne0$,
so $Q$ is an analytic function that isn't identically zero. Therefore,
$Q$ has at most finitely many zeroes in $\t$, proving the lemma.
\end{proof}

\paragraph*{Abundance of fully regular trigonometric polynomials.}

Since the condition for full regularity is coalgebraic, it is enough
to show it is nonempty to see that it is, in fact, dense in $\trigs$.
This is the content of the following:
\begin{lem}
There exists a $T\in\trigs$ that is fully regular.\end{lem}
\begin{proof}
We will show that
\[
T\left(x_{1},\dots,x_{d}\right)\coloneqq\sum_{j=1}^{d}\sin\left(2\pi Dx_{j}\right)+A,
\]
where $A>d$ is constant, satisfies the condition for nodal regularity
and the condition for critical set regularity. Since the restriction
of $T$ to a hyperplane of the form $\left\{ x_{j}=0\right\} $ takes
exactly the same form as $T$ with dimension smaller by $1$, this
is enough to imply that $T$ is fully regular.

Denote by $C_{D}\left(c,s\right)$ and $S_{D}\left(c,s\right)$ the
algebraizations of $\cos\left(2\pi Dx\right)$ and $\sin\left(2\pi Dx\right)$.
These are homogeneous polynomials of degree $D$ for which it is a
simple exercise to prove:
\selectlanguage{british}%
\begin{enumerate}[labelindent=\parindent,leftmargin=*,label=({\alph*})]
\item \foreignlanguage{american}{\label{enu:C^2+S^2=00003Dc^2+s^2}$\left(C_{D}\left(c,s\right)\right)^{2}+\left(S_{D}\left(c,s\right)\right)^{2}=\left(c^{2}+s^{2}\right)^{D}$.}
\selectlanguage{american}%
\item \label{enu:C=00003DS=00003D0->c=00003Ds=00003D0}The only solution
in $c,s\in\c$ of $C_{D}\left(c,s\right)=S_{D}\left(c,s\right)=0$
is $c=s=0$.
\item \label{enu:determinant-property}$\det\left(\begin{array}{cc}
\frac{\partial C_{D}}{\partial c}\left(c,s\right) & \frac{\partial C_{D}}{\partial s}\left(c,s\right)\\
c & s
\end{array}\right)=D\,S_{D}\left(c,s\right)$.
\end{enumerate}
\selectlanguage{american}%
The algebraization of $T,\nabla T$ is the following system of $2d+1$
polynomials in $2d$ variables $c_{1},s_{1},\dots,c_{d},s_{d}$:
\begin{equation}
\begin{array}{ccc}
\sum_{j=1}^{d}S_{D}\left(c_{j},s_{j}\right)+A, & \underbrace{2\pi D\,C_{D}\left(c_{j},s_{j}\right)}_{1\le j\le d}, & \underbrace{c_{j}^{2}+s_{j}^{2}-1}_{1\le j\le d}\end{array}\label{eq:nodal-regularity-system}
\end{equation}

The condition for nodal regularity (Condition \ref{cond:nodal-regularity})
requires this system to have no common nontrivial zeroes after homogenization.
We introduce a homogenizing variable $z_{0}$, and split to two cases:
$z_{0}=0$ and $z_{0}\ne0$. Since $S_{D}$ and $C_{D}$ are already
homogeneous of degree $D$, letting $z_{0}=0$, system (\ref{eq:nodal-regularity-system})
becomes:
\begin{equation}
\begin{array}{ccc}
\sum_{j=1}^{d}S_{D}\left(c_{j},s_{j}\right), & \underbrace{2\pi D\,C_{D}\left(c_{j},s_{j}\right)}_{1\le j\le d}, & \underbrace{c_{j}^{2}+s_{j}^{2}}_{1\le j\le d}\end{array}\label{eq:nodal-regularity-system-homogeneous}
\end{equation}

If $c_{1},s_{1},\ldots,c_{d},s_{d}$ is any zero of the system (\ref{eq:nodal-regularity-system-homogeneous}),
then $C_{D}\left(c_{j},s_{j}\right)=0$ and $c_{j}^{2}+s_{j}^{2}=0$
for $j=1,\ldots,d$, so by property \ref{enu:C^2+S^2=00003Dc^2+s^2}
above, $S_{D}\left(c_{j},s_{j}\right)=0$; by property \ref{enu:C=00003DS=00003D0->c=00003Ds=00003D0},
$c_{j}=s_{j}=0$ for $j=1,\ldots,d$, i.e.\ the zero must be trivial.

The other case is $z_{0}\ne0$, and it suffices to search for solutions
with $z_{0}=1$, i.e.\ solutions to the original system (\ref{eq:nodal-regularity-system}).
If $c_{1},s_{1},\ldots,c_{d},s_{d}$ is any zero of the system (\ref{eq:nodal-regularity-system}),
then $C_{D}\left(c_{j},s_{j}\right)=0$ and $c_{j}^{2}+s_{j}^{2}=1$
for $j=1,\ldots,d$. By property \ref{enu:C^2+S^2=00003Dc^2+s^2},
$\left|S_{D}\left(c_{j},s_{j}\right)\right|=1$ for $j=1,\ldots,d$.
But then $\sum_{j=1}^{d}S_{D}\left(c_{j},s_{j}\right)+A$ cannot be
zero, since $A>d$. Therefore there are no solutions with $z_{0}\ne0$.

To check the condition for critical set regularity (Condition \ref{cond:critical-regularity}),
we first write the polynomials in the algebraization of $\nabla T$
(as in system (\ref{eq:nodal-regularity-system}), excluding the first
polynomial) in the following order:
\begin{equation}
\begin{array}{ccccc}
2\pi D\,C_{D}\left(c_{1},s_{1}\right), & c_{1}^{2}+s_{1}^{2}-1, & \ldots, & 2\pi D\,C_{D}\left(c_{d},s_{d}\right), & c_{d}^{2}+s_{d}^{2}-1\end{array}\label{eq:algebraization-of-gradient}
\end{equation}

This system should satisfy the condition for finiteness (Condition
\ref{cond:finiteness}). We have already seen that its homogenization
has no zero with $z_{0}=0$, and it remains to check that it has no
zero common with its Jacobian. This Jacobian may be computed using
property \ref{enu:determinant-property}, and it equals $\left(4\pi D^{2}\right)^{d}\prod_{j=1}^{d}S_{D}\left(c_{j},s_{j}\right)$.
Thus, for any such common zero $\left(c_{1},s_{1},\ldots,c_{d},s_{d}\right)$,
there exists some $j$ such that $S_{D}\left(c_{j},s_{j}\right)=0$,
but also $C_{D}\left(c_{j},s_{j}\right)=0$ and $c_{j}^{2}+s_{j}^{2}=1$;
this is impossible by property \ref{enu:C^2+S^2=00003Dc^2+s^2}.
\end{proof}

\paragraph*{Proof of the main proposition for fully regular $T$.}

We may now prove a weak version of Proposition \ref{prop:trig-polynomials-nodal-set-result},
assuming the full regularity condition; this assumption will later
be lifted.

Background for the proof: Denote by $\Td$ the ``round'' torus -
the $d$-dimensional submanifold of $\mathbb{R}^{2d}$ defined by
the equations $c_{j}^{2}+s_{j}^{2}=1$, $j=1,\ldots,d$. Let $\varphi\colon\td\to\Td$
be the natural diffeomorphism $\varphi\left(x_{1},\ldots,x_{d}\right)\coloneqq\left(\cos\left(2\pi x_{1}\right),\sin\left(2\pi x_{1}\right),\ldots,\cos\left(2\pi x_{d}\right),\sin\left(2\pi x_{d}\right)\right)$.
This diffeomorphism is not an isometry, but it induces a strongly
equivalent metric; that is, there exist constants $\alpha,\beta>0$
such that for any $x,y\in\mathbb{T}^{d}$, $\alpha\dist\left(x,y\right)\le\left|\varphi\left(x\right)-\varphi\left(y\right)\right|\le\beta\dist\left(x,y\right).$

We will need a little background in integral geometry. For any compact,
connected set $K\subset\rk$, denote by $W_{j}\left(K\right)\coloneqq\max_{x,y\in K}\left|x_{j}-y_{j}\right|$
the \emph{width of $K$ along the $j$th axis}, and by $W\left(K\right)\coloneqq\frac{1}{k}\sum_{j=1}^{k}W_{j}\left(K\right)$
the \emph{average width}\footnote{This is a simpler version of the \emph{mean width}, defined similarly
as an integral average of support functions.}. It is an easy exercise to show that $\diam K\ll W\left(K\right)$,
and also that $W_{j}\left(K\right)=\Int[-\infty][\infty]{I\left(K,\left\{ x_{j}=\gamma\right\} \right)}{\gamma}$,
where $I\left(A,B\right)$ is the intersection indicator function
($1$ if $A\cap B\ne\varnothing$, $0$ otherwise). This is all we
need; for more on the subject of integral geometry, see \cite{Klain-Rota-book-Introduction-to-geometric-probability}.
\begin{lem}
\label{lem:fully-regular-trig-polys-admit-results-of-main-proposition}Suppose
$T\in\trigs$ is fully regular. Then $N\left(T\right)\ll D^{d}$ and
$\sum_{\Gamma\in\Z\left(T\right)}\diam\Gamma\ll D^{d-1}$.\end{lem}
\begin{proof}
By Lemma \ref{lem:cond-finiteness-results}, the number of critical
points of $T$ in $\td$ is at most $O\left(D^{d}\right)$. There
is at least one critical point in each nodal domain, since each nodal
domain has a (necessarily nonzero) minimum or a maximum point by compactness.
Therefore there are $O\left(D^{d}\right)$ nodal domains, and by Proposition
\ref{prop:nodal-components-similar-to-nodal-domains}, $O\left(D^{d}\right)$
nodal components, proving the first part.

For the second part, let $\Gamma\in\Z\left(T\right)$. We have:
\[
\diam\Gamma\ll\diam\varphi\left(\Gamma\right)\ll W\left(\varphi\left(\Gamma\right)\right)=\frac{1}{2d}\sum_{j=1}^{2d}\Int[-\infty][\infty]{I\left(\varphi\left(\Gamma\right),\left\{ x_{j}=\gamma\right\} \right)}{\gamma}.
\]

Since intersections are preserved by the bijection $\varphi$, the
integrand $I\left(\varphi\left(\Gamma\right),\left\{ x_{j}=\gamma\right\} \right)$
can be written equivalently as $I\left(\Gamma,H_{j,\gamma}\right)$,
where:
\[
H_{j,\gamma}\coloneqq\varphi^{-1}\left(\left\{ x\in\R^{2d}:x_{j}=\gamma\right\} \right)=\begin{cases}
\left\{ x\in\td:\cos\left(2\pi x_{j}\right)=\gamma\right\}  & j\text{ odd}\\
\left\{ x\in\td:\sin\left(2\pi x_{j}\right)=\gamma\right\}  & j\text{ even}
\end{cases}
\]

We may restrict the integral's limits to $\left|\gamma\right|<1$
and get: $\diam\Gamma\ll\sum_{j=1}^{2d}\Int[-1][1]{I\left(\Gamma,H_{j,\gamma}\right)}{\gamma}$.
Summing over all $\Gamma\in\Z\left(T\right)$ (a finite sum), we get:
\[
\sum_{\Gamma\in\Z\left(T\right)}\diam\Gamma\ll\sum_{j=1}^{2d}\Int[-1][1]{\left(\text{\# of }\Gamma\in\Z\left(T\right)\text{ which intersect }H_{j,\gamma}\right)}{\gamma}.
\]

The number of nodal components that intersect some hyperplane is clearly
bounded by the number of nodal components of the restriction of $T$
to this hyperplane. $H_{j,\gamma}$ is the union of two hyperplanes
of the form $\left\{ x_{j}=t\right\} $, and by Lemma \ref{lem:improved-regularity}
and the first part of this lemma, for almost all $\gamma$ the number
of nodal components in the restriction is $O\left(D^{d-1}\right)$.
This gives the required bound.
\end{proof}

\paragraph*{Proof of the main proposition without the regularity condition.}
\begin{proof}[Proof of Proposition \ref{prop:trig-polynomials-nodal-set-result}.]
Let $\alpha,\beta>0$ be such that $\left|T\left(x\right)\right|>\alpha$
or $\left|\nabla T\left(x\right)\right|>\beta$ for any $x\in\td$,
under the constraint that $\alpha/\beta<\left(2D\right)^{-1}$ (see
Remark \ref{rem:about-nodal-stability}). Since the condition for
full regularity is dense in $\trigs$, it is possible to choose a
perturbation $P\in\trigs$ with $\max_{\td}\left|P\right|<\alpha/2$
and $\max_{\td}\left|\nabla P\right|<\beta/2$ such that $T+P$ is
fully regular.

Applying Proposition \ref{prop:smooth-perturbation} with the function
$T$ and the perturbation $P$ gives $N\left(T\right)=N\left(T+P\right)$.
By Lemma \ref{lem:fully-regular-trig-polys-admit-results-of-main-proposition},
$N\left(T+P\right)\ll D^{d}$; therefore, $N\left(T\right)\ll D^{d}$.

By Proposition \ref{prop:smooth-perturbation}, each component $\Gamma\in\Z\left(T\right)$
generates a component $\tilde{\Gamma}\in\Z\left(T+P\right)$ that
satisfies:
\[
\diam\Gamma\le\frac{2\alpha}{\beta}+\diam\tilde{\Gamma}\le\frac{1}{D}+\diam\tilde{\Gamma}.
\]

Therefore: 
\[
\sum_{\Gamma\in\Z\left(T\right)}\diam\Gamma\le\frac{N\left(T\right)}{D}+\quad\sum_{\mathclap{\Gamma\in\Z\left(T+P\right)}}\quad\diam\Gamma.
\]

We've established $N\left(T\right)\ll D^{d}$, and by Lemma \ref{lem:fully-regular-trig-polys-admit-results-of-main-proposition}
we have $\sum_{\Gamma\in\Z\left(T+P\right)}\diam\Gamma\ll D^{d-1}$;
therefore, $\sum_{\Gamma\in\Z\left(T\right)}\diam\Gamma\ll D^{d-1}$.
\end{proof}

\section{\label{sec:proof-of-main-theorem}Proof of Theorem \ref{thm:exponential-concentration}}

\subsection{Tools used in the proof}
\begin{prop}
\label{prop:nodal-domain-area}Any nodal domain $\Omega$ of any $f\in\hilbert$
satisfies $\vol\left(\Omega\right)\gg L^{-d}$\textup{.}
\end{prop}
Proposition \ref{prop:nodal-domain-area} follows immediately from
the classical Faber-Krahn inequality (note that the first Dirichlet
eigenvalue of any nodal domain $\Omega$ is $4\pi^{2}L^{2}$). See,
for instance, \cite[Chapter IV]{Chavel-book-Eigenvalues-in-Riemannian-geometry}.
\begin{prop}
\label{prop:local-bounds}Suppose $f\colon\rd\to\R$ is smooth and
satisfies $\Delta f+4\pi^{2}L^{2}f=0$. Let $x_{0}\in\rd$ and $r>0$.
Then there is a constant $C=C\left(r,d\right)>0$ such that:
\begin{align}
\left|f\left(x_{0}\right)\right|^{2} & \le CL^{d}\Int[B^{d}\left(x_{0},r/L\right)]{\left|f\left(x\right)\right|^{2}}x\label{eq:bound-function}\\
\left|\nabla f\left(x_{0}\right)\right|^{2} & \le CL^{d+2}\Int[B^{d}\left(x_{0},r/L\right)]{\left|f\left(x\right)\right|^{2}}x\label{eq:bound-gradient}\\
\left|\nabla\nabla f\left(x_{0}\right)\right|^{2} & \le CL^{d+4}\Int[B^{d}\left(x_{0},r/L\right)]{\left|f\left(x\right)\right|^{2}}x\label{eq:bound-hessian}
\end{align}

\end{prop}
Proposition \ref{prop:local-bounds} is a local property of functions
on $\rd$, and as such, it may be applied directly to functions in
$\hilbert$, viewed as functions on $\rd$ periodic extension. These
local bounds are special cases of very general classical local bounds
on solutions of PDEs (see, for instance, \cite[Chapter 8]{Gilbarg-Trudinger-book}),
but for the sake of completeness, we provide a simple and easily readable
proof for our case in Appendix \ref{sec:appendix-additional-proofs}.
\begin{prop}
\label{prop:concentration-of-norm-of-fL}$\p\left\{ \left\Vert f_{L}\right\Vert >2\right\} \le\e^{-c\dim\hilbert}$
for some absolute constant $c>0$.
\end{prop}
Considering $f_{L}$ as a random vector in $\rn$ (where $n=\dim\hilbert$)
that is distributed like $\frac{X}{\sqrt{n}}$, where $X$ has standard
multivariate normal distribution in $\rn$, Proposition \ref{prop:concentration-of-norm-of-fL}
may be formulated equivalently as $\p\left\{ \left|X\right|>2\sqrt{n}\right\} \le\e^{-cn}$.
This is a special case of Bernstein's classical inequalities and it
may be proven by applying Chebyshev's inequality on $\exp\left(\frac{1}{4}\left|X\right|^{2}\right)$.
We omit the details.

Again considering $f_{L}$ as a random vector in $\rn$, the next
proposition is a form of the Gaussian isoperimetric inequality (see
\cite{Sudakov-Tsirelson}, \cite{Borell}).
\begin{prop}
\label{prop:isoperimetric-fL}Let $F\subset\hilbert$ and for any
$\rho>0$, denote $F_{+\rho}\coloneqq\left\{ f\in\hilbert:\dist\left(f,F\right)\le\rho\right\} $.

Suppose that $\mathbb{P}\left(F_{+\rho}\right)\le\frac{3}{4}$. Then
$\mathbb{P}\left(F\right)\le C\e^{-c\rho^{2}\dim\hilbert}$ for some
absolute constants $C,c>0$.
\end{prop}
Next, recall that a smooth function is said to have a \emph{stable
nodal set} if it doesn't have zeroes in common with its gradient (Definition
\ref{def:stable-nodal-set}). The following proposition follows either
from Bulinskaya's lemma \cite[Lemma 11.2.10]{Adler-Taylor-book-Random-fields-and-geometry}
or from \cite[Lemma 2.3]{Oravecz-Rudnick-Wigman-the-Leray-measure-of-nodal-sets}.
\begin{prop}
\label{prop:a.s.-stable-nodal-set}Almost surely, $f_{L}$ has a stable
nodal set.
\end{prop}

\subsection{\label{sub:concentration-implies-concentration}Concentration around
the median implies concentration around the mean and limiting mean}

We begin by showing that in Theorem \ref{thm:exponential-concentration},
the first part implies the second and third parts. Throughout the
remainder of this section, we denote $m_{L}\coloneqq\median\left\{ N_{L}/L^{d}\right\} $.

Let $X_{L}\coloneqq N_{L}/L^{d}-m_{L}$. By the first part of the
theorem, $\mathbb{P}\left\{ \left|X_{L}\right|>\varepsilon\right\} \le C\left(\varepsilon\right)\e^{-c\left(\varepsilon\right)\dim\hilbert}$.
The random variables $X_{L}$ are a.s.\ uniformly bounded, since
$N_{L}/L^{d}$ are a.s.\ uniformly bounded by Courant's nodal domain
theorem (or alternatively by Proposition \ref{prop:trig-polynomials-nodal-set-result}).
By the law of total expectation:
\begin{align*}
\left|\E\left\{ X_{L}\right\} \right| & \le\left|\E\left\{ X_{L}|X_{L}\ge\varepsilon\right\} \right|\cdot\p\left\{ X_{L}\ge\varepsilon\right\} \\
 & \phantom{\le\left|\E\left\{ X_{L}|X_{L}\ge\varepsilon\right\} \right|}+\left|\E\left\{ X_{L}|\left|X_{L}\right|<\varepsilon\right\} \right|\cdot\p\left\{ \left|X_{L}\right|<\varepsilon\right\} +\left|\E\left\{ X_{L}|X_{L}\le-\varepsilon\right\} \right|\cdot\p\left\{ X_{L}\le-\varepsilon\right\} .
\end{align*}
The first and third terms are bounded by $C\left(\varepsilon\right)\e^{-c\left(\varepsilon\right)\dim\hilbert}$
multiplied by the a.s.\ uniform bound on $X_{L}$, while the second
term is bounded by $\varepsilon$. We may assume that $\dim\hilbert$
is large enough (see Remark \ref{rem:understanding-the-main-result})
such that the sum of the first and third terms is smaller than $\varepsilon$,
and get $\left|\E\left\{ X_{L}\right\} \right|\le2\varepsilon$. Therefore,
by the triangle inequality:
\[
\p\left\{ \left|\frac{N_{L}}{L^{d}}-\E\left\{ \frac{N_{L}}{L^{d}}\right\} \right|>3\varepsilon\right\} =\p\left\{ \left|X_{L}-\E\left\{ X_{L}\right\} \right|>3\varepsilon\right\} \le\p\left\{ \left|X_{L}\right|>\varepsilon\right\} \le C\left(\varepsilon\right)\e^{-c\left(\varepsilon\right)\dim\hilbert}.
\]
This proves the second part of the theorem. For the third part, let
$L$ be large enough such that $\left|\E\left\{ \frac{N_{L}}{L^{d}}\right\} -\nu\right|\le\varepsilon$.
Then, by the triangle inequality:
\[
\p\left\{ \left|\frac{N_{L}}{L^{d}}-\nu\right|>4\varepsilon\right\} \le\p\left\{ \left|\frac{N_{L}}{L^{d}}-\E\left\{ \frac{N_{L}}{L^{d}}\right\} \right|>3\varepsilon\right\} \le C\left(\varepsilon\right)\e^{-c\left(\varepsilon\right)\dim\hilbert}.
\]

\subsection{The exceptional set of instability}

We now show that the concentration around the median follows from
the existence of a small \emph{exceptional set of instability}, which
we will later construct. This is an exponentially small set $E\subset\hilbert$
such that outside this set, the number of nodal components is stable
under sufficiently small perturbations.
\begin{prop}
\label{prop:exceptional-set-implies-theorem}Suppose that for every
$\epsilon>0$, there exist $\rho=\rho\left(\epsilon\right)>0$ and
$\tau=\tau\left(\varepsilon\right)>0$ such that for every $L$, there
exists an ``exceptional set of instability'' $E=E\left(\epsilon,L\right)\subset\mathcal{H}_{L}$
satisfying two conditions:
\begin{enumerate}
\item ($f_{L}$ has exponentially small probability to be exceptional.)
For some constants $C\left(\varepsilon\right),c>0$,
\begin{equation}
\mathbb{P}\left(E\right)\le\min\left\{ \frac{1}{4},C\left(\epsilon\right)\e^{-c\tau^{2}\dim\hilbert}\right\} .\label{eq:bound-on-P(E)}
\end{equation}

\item ($N$ is lower semi-continuous for non-exceptional functions.) For
any $f\in\hilbert\setminus E$ and $g\in\hilbert$ such that $\left\Vert g\right\Vert \le\rho$,
\begin{equation}
N\left(f+g\right)\ge N\left(f\right)-\epsilon L^{d}.\label{eq:N-is-lsc-outside-E}
\end{equation}

\end{enumerate}
Then the first part of Theorem \ref{thm:exponential-concentration}
holds with constant $c\left(\varepsilon\right)$ proportional to $\min\left\{ \rho^{2},\tau^{2}\right\} $.\end{prop}
\begin{proof}
Notice that $\left\{ f\in\hilbert:\left|N\left(f\right)/L^{d}-m_{L}\right|>\epsilon\right\} =F\cup G$,
where:
\begin{align*}
F & =\left\{ f\in\mathcal{H}_{L}:N\left(f\right)>\left(m_{L}+\epsilon\right)L^{d}\right\} \\
G & =\left\{ f\in\mathcal{H}_{L}:N\left(f\right)<\left(m_{L}-\epsilon\right)L^{d}\right\} 
\end{align*}

First, we bound $\p\left(F\right)$. Let $h\in\left(F\setminus E\right)_{+\rho}$;
that is, $h=f+g$ where $f\in F\setminus E$ and $g\in\hilbert$ satisfies
$\left\Vert g\right\Vert \le\rho$. Then:
\[
N\left(h\right)=N\left(f+g\right)\stackrel{\left(f\notin E,\left\Vert g\right\Vert \le\rho\right)}{\ge}N\left(f\right)-\varepsilon L^{d}\stackrel{\left(f\in F\right)}{>}m_{L}L^{d}.
\]
Therefore $\left(F\setminus E\right)_{+\rho}\subset\left\{ h\in\mathcal{H}_{L}:N\left(h\right)/L^{d}>m_{L}\right\} $,
the latter set having probability at most $\frac{1}{2}$.

Proposition \ref{prop:isoperimetric-fL} gives $\mathbb{P}\left(F\setminus E\right)\le C\e^{-c\rho^{2}\dim\hilbert}$.
Together with $\p\left(E\right)\le C\left(\varepsilon\right)\e^{-c\tau^{2}\dim\hilbert}$,
we have $\mathbb{P}\left(F\right)\le C\left(\epsilon\right)\e^{-c\left(\epsilon\right)\dim\hilbert}$
with $c\left(\varepsilon\right)$ proportional to $\min\left\{ \rho^{2},\tau^{2}\right\} $.

Next, we bound $\p\left(G\right)$. Let $h\in G_{+\rho}\setminus E$;
that is, $h=f+g$ where $f\in G$, $g\in\hilbert$ satisfies $\left\Vert g\right\Vert \le\rho$,
and $h\notin E$. Then:
\[
\left(m_{L}-\varepsilon\right)L^{d}\stackrel{\left(f\in G\right)}{>}N\left(f\right)=N\left(h-g\right)\stackrel{\left(h\notin E,\left\Vert g\right\Vert \le\rho\right)}{\ge}N\left(h\right)-\varepsilon L^{d}.
\]
Therefore $G_{+\rho}\setminus E\subset\left\{ h\in\mathcal{H}_{L}:N\left(h\right)/L^{d}<m_{L}\right\} $,
the latter set having probability at most $\frac{1}{2}$. Thus $\mathbb{P}\left(G_{+\rho}\right)\le\frac{3}{4}$,
and applying Proposition \ref{prop:isoperimetric-fL} again gives
$\mathbb{P}\left(G\right)\le C\e^{-c\rho^{2}\dim\hilbert}$.
\end{proof}

\subsection{Construction of the exceptional set \texorpdfstring{$E$}{E}}

We now present the construction of the set $E=E\left(\epsilon,L\right)\subset\hilbert$,
so throughout this part of the paper, $\epsilon$ and $L$ are fixed.
We introduce new small parameters $0<\alpha,\beta,\delta<1$ and one
large parameter $R>2$ that all depend only on $\epsilon$ in a way
that will be determined later.

Cover the torus $\mathbb{T}^{d}$ by as few as possible closed balls
$\left\{ B_{j}\right\} $ of radius $RL^{-1}$ (their amount is $O\left(L^{d}R^{-d}\right)$).
We will later refer to $2B_{j}$, $3B_{j}$ and $4B_{j}$; these are
balls with the same center as $B_{j}$ and radius multiplied by $2$,
$3$ and $4$ respectively. We require the cover to satisfy a bounded
multiplicity condition: For any point $x\in\td$, the amount of balls
$4B_{j}$ that cover $x$ is $O\left(1\right)$.

For any $f\in\mathcal{H}_{L}$, we say that $3B_{j}$ is \emph{an
unstable ball with respect to $f$} if there exists a point $x\in3B_{j}$
such that $\left|f\left(x\right)\right|\le\alpha$ and $\left|\nabla f\left(x\right)\right|\le\beta L$.

Finally, define $E\subset\hilbert$ as the set of functions for which
the number of unstable balls exceeds $\delta L^{d}$. It is easy to
verify that $E$ is measurable.

\subsection{Proof that \texorpdfstring{$E$}{E} has exponentially small probability}

We will now prove (\ref{eq:bound-on-P(E)}) under certain assumptions
that will arise from the proof. We introduce two new small parameters
$0<\gamma,\tau<1$ that depend only on $\varepsilon$ in a way that
will be determined later.

It suffices to prove (\ref{eq:bound-on-P(E)}) for $\p\left(\tilde E\right)$
where $\tilde E\coloneqq E\cap\left\{ f\in\hilbert:\left\Vert f\right\Vert \le2\right\} $
instead of $E$. The leftover set $E\setminus\tilde E$ may be discarded
by Proposition \ref{prop:concentration-of-norm-of-fL} as it has comparatively
negligible probability.

Thus, let $f\in E$ with $\left\Vert f\right\Vert \le2$. In each
ball $3B_{j}$ that is unstable with respect to $f$, fix a point
$x_{j}\in3B_{j}$ such that $\left|f\left(x_{j}\right)\right|\le\alpha$
and $\left|\nabla f\left(x_{j}\right)\right|\le\beta L$.

Using the local bound from Proposition \ref{prop:local-bounds} with
arbitrary $x_{0}\in B\left(x_{j},\gamma L^{-1}\right)$ and $r=1$,
we have:
\begin{equation}
\sup_{x_{0}\in B\left(x_{j},\gamma L^{-1}\right)}\left|\nabla\nabla f\left(x_{0}\right)\right|^{2}\ll\sup_{x_{0}\in B\left(x_{j},\gamma L^{-1}\right)}L^{d+4}\Int[B\left(x_{0},L^{-1}\right)]{\left|f\left(x\right)\right|^{2}}x\ll L^{d+4}\Int[B\left(x_{j},2L^{-1}\right)]{\left|f\left(x\right)\right|^{2}}x.\label{eq:E-is-small-1}
\end{equation}
Since $R>2$, we have $B\left(x_{j},2L^{-1}\right)\subset4B_{j}$,
and the amount of balls $B\left(x_{j},2L^{-1}\right)$ that any point
$x\in\td$ belongs to is $O\left(1\right)$. Therefore, summing over
all balls:
\begin{equation}
\sum_{j}\Int[B\left(x_{j},2L^{-1}\right)]{\left|f\left(x\right)\right|^{2}}x\ll\left\Vert f\right\Vert ^{2}\ll1.\label{eq:E-is-small-2}
\end{equation}

Plugging (\ref{eq:E-is-small-2}) into (\ref{eq:E-is-small-1}) yields
$\sum_{j}\sup_{B\left(x_{j},\gamma L^{-1}\right)}\left|\nabla\nabla f\right|^{2}\ll L^{d+4}$.
There are at least $\delta L^{d}$ summands, so the average summand
is $O\left(L^{4}\delta^{-1}\right)$. Multiplying the hidden constant
by $4$, at least a proportion $\frac{3}{4}$ of all indexes $j$
satisfy:
\begin{equation}
\sup_{B\left(x_{j},\gamma L^{-1}\right)}\left|\nabla\nabla f\right|\ll L^{2}\delta^{-1/2}.\label{eq:bound-on-sup-grad-grad-f}
\end{equation}

Now, introduce a perturbation $g\in\hilbert$ with $\left\Vert g\right\Vert \le\tau$.
By the same argument as above, using Proposition \ref{prop:local-bounds}
and summing over $j$, we have $\sum_{j}\sup_{B\left(x_{j},\gamma L^{-1}\right)}\left|g\right|^{2}\ll L^{d}\tau^{2}$,
and at least a proportion $\frac{3}{4}$ of all indexes $j$ satisfy:
\begin{equation}
\sup_{B\left(x_{j},\gamma L^{-1}\right)}\left|g\right|\ll\tau\delta^{-1/2}.\label{eq:bound-on-sup-g}
\end{equation}
Using this argument for the third time, now with $\nabla g$, we have
$\sum_{j}\sup_{B\left(x_{j},\gamma L^{-1}\right)}\left|\nabla g\right|^{2}\ll L^{d+2}\tau^{2}$,
so at least a proportion $\frac{3}{4}$ of all indexes $j$ satisfy:
\begin{equation}
\sup_{B\left(x_{j},\gamma L^{-1}\right)}\left|\nabla g\right|\ll\tau L\delta^{-1/2}.\label{eq:bound-on-sup-grad-g}
\end{equation}

Thus, at least a proportion $\frac{1}{4}$ of all indexes $j$, i.e.\ at
least $\frac{1}{4}\delta L^{d}$ indexes, satisfy all three (\ref{eq:bound-on-sup-grad-grad-f}),
(\ref{eq:bound-on-sup-g}) and (\ref{eq:bound-on-sup-grad-g}). For
such indexes $j$, applying Taylor's formula on $f$ and on $\nabla f$,
we get:
\begin{align}
\sup_{B\left(x_{j},\gamma L^{-1}\right)}\left|f\right| & \le\alpha+\beta\gamma+C\gamma^{2}\delta^{-1/2}\label{eq:bound-on-sup-f}\\
\sup_{B\left(x_{j},\gamma L^{-1}\right)}\left|\nabla f\right| & \le\left(\beta+C\gamma\delta^{-1/2}\right)L\label{eq:bound-on-sup-grad-f}
\end{align}
Here, $C$ is the hidden constant from (\ref{eq:bound-on-sup-grad-grad-f}).

By summing (\ref{eq:bound-on-sup-g}) + (\ref{eq:bound-on-sup-f})
and (\ref{eq:bound-on-sup-grad-g}) + (\ref{eq:bound-on-sup-grad-f}),
we get:
\begin{align*}
\sup_{B\left(x_{j},\gamma L^{-1}\right)}\left|f+g\right| & \le\underbrace{\alpha+\beta\gamma+C\delta^{-1/2}\left(\gamma^{2}+\tau\right)}_{A}\\
\sup_{B\left(x_{j},\gamma L^{-1}\right)}\left|\nabla f+\nabla g\right| & \le\underbrace{\left(\beta+C\delta^{-1/2}\left(\gamma+\tau\right)\right)}_{B}L
\end{align*}
Therefore, $f+g\in U$, where:
\[
U\coloneqq\left\{ h\in\mathcal{H}_{L}:\vol\left\{ x\in\td:\left|h\left(x\right)\right|\le A,\left|\nabla h\left(x\right)\right|\le BL\right\} \ge c\delta\gamma^{d}\right\} .
\]
For every fixed $x\in\td$, we have $\p\left\{ \left|f_{L}\left(x\right)\right|\le A,\left|\nabla f_{L}\left(x\right)\right|\le BL\right\} \ll AB^{d}$
due to the independence of the random variable $f_{L}\left(x\right)$
and the random vector $\nabla f_{L}\left(x\right)$ and the fact that
they have bounded densities in $\R$ and $\rd$, respectively. By
Fubini's theorem, we also have:
\[
\E\left\{ \vol\left\{ x\in\td:\left|h\left(x\right)\right|\le A,\left|\nabla h\left(x\right)\right|\le BL\right\} \right\} \ll AB^{d}.
\]
Thus, if we assume $AB^{d}\le C\delta\gamma^{d}$ for an appropriate
constant $C$, we get by Chebyshev's inequality that $\p\left(U\right)\le\frac{1}{2}$
(or any other constant smaller than $\frac{1}{2}$; since $\tilde E\subset U$,
we may decrease this constant to ensure $\p\left(E\right)\le\frac{1}{4}$).
Since furthermore $\p\left(\tilde E_{+\tau}\right)\le\frac{1}{2}$,
by Proposition \ref{prop:isoperimetric-fL} we get $\mathbb{P}\left(\widetilde{E}\right)\le C\e^{-c\tau^{2}\dim\hilbert}$.

To conclude, we have showed (\ref{eq:bound-on-P(E)}) under the following
assumption:
\begin{assumption}
$\left(\alpha+\beta\gamma+C\delta^{-1/2}\left(\gamma^{2}+\tau\right)\right)\left(\beta+C\delta^{-1/2}\left(\gamma+\tau\right)\right)^{d}\ll\delta\gamma^{d}$.
\end{assumption}

\subsection{Proof that \texorpdfstring{$N$}{N} is lower semi-continuous outside
of \texorpdfstring{$E$}{E}}

Next, we prove (\ref{eq:N-is-lsc-outside-E}) for $f\in\hilbert\setminus E$
and $g\in\hilbert$ with $\left\Vert g\right\Vert \le\rho$, where
$0<\rho<1$ is a new parameter depending only on $\varepsilon$ in
a way that will be determined later. Again, this proof will require
a few assumptions.

We may assume $f$ has a stable nodal set (this is a.s.\ true). For
each nodal component $\Gamma\in\Z\left(f\right)$, we pick one ball
$B_{j}$ that intersects it, and call it \emph{the intersecting ball
of $\Gamma$}. Now, assume three conditions: \emph{(a)} $\diam\Gamma\le RL^{-1}$;
\emph{(b)} $3B_{j}$ is a stable ball for $f$; and \emph{(c)} $\sup_{3B_{j}}\left|g\right|<\alpha$.
A component $\Gamma\in\Z\left(f\right)$ that satisfies these conditions
(with $B_{j}$ being its intersecting ball) is said to be a \emph{controllable
component}.

Since the radius of $B_{j}$ if $RL^{-1}$, we have $\Gamma\subset2B_{j}$,
thus $\Gamma$ is $RL^{-1}$-separated from the boundary of $3B_{j}$;
under an additional assumption that $R\ge\alpha/\beta$, we have $\Gamma_{+\alpha/\left(\beta L\right)}\subset3B_{j}$.
Since $3B_{j}$ is a stable ball for $f$, we have $\left|f\left(x\right)\right|>\alpha$
or $\left|\nabla f\left(x\right)\right|>\beta L$ for any $x\in3B_{j}$,
and also $\left|g\left(x\right)\right|<\alpha$. By Proposition \ref{prop:perturbation-in-U},
$\Gamma$ generates a component $\tilde{\Gamma}\in\Z\left(f+g\right)$
that is also contained in $3B_{j}$, and the mapping $\Gamma\mapsto\tilde{\Gamma}$
is injective for all $\Gamma$ satisfying condition (a) above with
intersecting ball $B_{j}$. Thus, all controllable components $\Gamma$
generate different components $\tilde{\Gamma}$ under perturbation
by $g$, and their number does not decrease.

Thus, to prove (\ref{eq:N-is-lsc-outside-E}) it remains to show that
the number of components that are not controllable can be bounded
by $\epsilon L^{d}$.

First, by Proposition \ref{prop:trig-polynomials-nodal-set-result},
the number of components $\Gamma$ with $\diam\Gamma>RL^{-1}$ is
a.s.\ $O\left(L^{d}R^{-1}\right)$, and we get the required bound
with an additional assumption: $R^{-1}\ll\epsilon$.

Second, suppose $\diam\Gamma\le RL^{-1}$ (thus $\Gamma\subset2B_{j}$),
but $3B_{j}$ is an unstable ball for $f$. The number of components
$\Gamma$ that can fit into $2B_{j}$ is at most $O\left(R^{d}\right)$:
If the radius of $2B_{j}$ is greater than half, then this is trivial
(as there are no more than $O\left(L^{d}\right)$ components in general);
otherwise, by Proposition \ref{prop:nodal-components-similar-to-nodal-domains}
it suffices to bound the number of nodal domains, which has the required
bound by Proposition \ref{prop:nodal-domain-area}. Since $f\notin E$,
there are at most $\delta L^{d}$ unstable balls, so there are $O\left(\delta L^{d}R^{d}\right)$
components of this type, which has the required bound with an additional
assumption: $\delta R^{d}\ll\epsilon$.

Finally, suppose $k$ of the balls $3B_{j}$ satisfy $\sup_{3B_{j}}\left|g\right|\ge\alpha$.
For each such ball, fix $x_{j}\in3B_{j}$ for which $\left|g\left(x_{j}\right)\right|\ge\alpha$.
By Proposition \ref{prop:local-bounds}, we have:
\[
\alpha^{2}\le\left|g\left(x_{j}\right)\right|^{2}\ll L^{d}\Int[B\left(x_{j},2L^{-1}\right)]{\left|g\left(x\right)\right|^{2}}x.
\]
Since a constant proportion of the balls $B\left(x_{j},2L^{-1}\right)$
may be dropped leaving a disjoint collection of balls, we get by summing
over the remaining balls:
\[
k\ll L^{d}\alpha^{-2}\rho^{2}.
\]

We have seen that at most $O\left(R^{d}\right)$ components have $B_{j}$
as their intersecting ball, so the number of components for which
the intersecting ball satisfies $\sup_{3B_{j}}\left|g\right|\ge\alpha$
is at most $O\left(L^{d}R^{d}\alpha^{-2}\rho^{2}\right)$, which has
the required bound with an additional assumption: $R^{d}\alpha^{-2}\rho^{2}\ll\epsilon$.

To summarize, (\ref{eq:N-is-lsc-outside-E}) is proven assuming the
following assumptions:
\begin{assumption}
$R\ge\alpha/\beta$.
\end{assumption}

\begin{assumption}
$R^{-1}\ll\epsilon$.
\end{assumption}

\begin{assumption}
$\delta R^{d}\ll\epsilon$.
\end{assumption}

\begin{assumption}
$R^{d}\alpha^{-2}\rho^{2}\ll\epsilon$.
\end{assumption}

\subsection{Choice of parameters}

It remains to choose values for the small parameters $\alpha,\beta,\delta,\gamma,\tau,\rho$
and large parameter $R$ that were used in the previous three subsections;
firstly, to satisfy the five assumptions that arose from the proofs,
and secondly, to maximize the value $c\left(\varepsilon\right)\simeq\min\left\{ \rho^{2},\tau^{2}\right\} $
that appears in Theorem \ref{thm:exponential-concentration}.

The conditions to satisfy are asymptotic inequalities, so we express
each parameter asymptotically as a power of $\varepsilon$: $\alpha\simeq\varepsilon^{a}$,
$\beta\simeq\epsilon^{b}$, $\delta\simeq\epsilon^{2k}$, $\gamma\simeq\epsilon^{g}$,
$\tau\simeq\epsilon^{t}$, $\rho\simeq\epsilon^{h}$ and $R\simeq\epsilon^{-r}$.
Each of $a,b,k,g,t,h,r$ must be a positive real number, and the five
asymptotic inequalities above can be expressed as the following constraints
on the exponents:
\begin{align}
2k+dg & \le\min\left\{ a,b+g,2g-k,t-k\right\} +d\cdot\min\left\{ b,g-k,t-k\right\} \label{eq:ineq1}\\
b & \le a+r\label{eq:ineq2}\\
r & \ge1\label{eq:ineq3}\\
2k & \ge1+rd\label{eq:ineq4}\\
2h & \ge1+2a+rd\label{eq:ineq5}
\end{align}
That is, we want to find positive values that satisfy the above, and
\emph{minimize} $\max\left\{ h,t\right\} $.

First, note that $h$ is minimized simply by changing (\ref{eq:ineq5}),
the only inequality it appears in, to an equation, and we get $2h=1+2a+rd$.

Next, observe that in (\ref{eq:ineq1}), any choice of two minimal
values on the right hand side creates a simple inequality. There are
12 such inequalities, and the inequality (\ref{eq:ineq1}) is equivalent
to all 12 occurring simultaneously. Of those 12 inequalities, 6 are
constraints on $t$ (in the other 6, $t$ does not appear). In the
6 constraints on $t$, we find that increasing $b$, decreasing $k$
or decreasing $g$ either weakens or doesn't change the constraint
on minimizing $t$. Therefore, $b$ must be maximized and $k$ and
$g$ must be minimized. To maximize $b$, we may change the inequality
(\ref{eq:ineq2}) into an equation $b=a+r$, and to maximize $k$,
(\ref{eq:ineq4}) gives $2k=1+rd$.

Since $b+g=a+r+g>a$, we may drop the term $b+g$ from the first minimum
in the right hand side of (\ref{eq:ineq1}). This leaves only one
inequality that bounds $g$ from below - the one obtained by choosing
$2g-k$ as the first minimum and $g-k$ as the second. From this,
we get $2g=\left(3+d\right)k$.

Finally, decreasing $r$ decreases $h$ and does not affect any constraint
on $t$, so $r$ must be minimized. The only remaining constraint
on $r$ is (\ref{eq:ineq4}), so we get $r=1$.

Collecting and simplifying all of the above equations, we have:
\begin{align*}
b & =a+1\\
2k & =d+1\\
4g & =\left(d+1\right)\left(d+3\right)\\
2h & =d+1+2a\\
r & =1
\end{align*}

Our target is to minimize $\max\left\{ h,t\right\} $. Decreasing
$a$ decreases $h$, but tightens the constraints in (\ref{eq:ineq1})
on minimizing $t$. Therefore, we set $t=h$ and find minimal value
of $a$ by the constraint (\ref{eq:ineq1}), which may now be written
as:
\[
\left(d+1\right)\left(d^{2}+3d+4\right)\le\min\left\{ 4a,2\left(d+1\right)\left(d+2\right)\right\} +d\cdot\min\left\{ 4a,\left(d+1\right)^{2}\right\} .
\]

The minimal solution to this inequality is:
\[
2a=\left(d+1\right)\left(d+2\right),
\]
and we can see that all constraints are satisfied, with $2h=2t=\left(d+2\right)^{2}-1$;
so, $c\left(\epsilon\right)\simeq\epsilon^{\left(d+2\right)^{2}-1}$.

\appendix

\section{\label{sec:appendix-equidistribution}Equidistribution of lattice
points on spheres}

For completeness, we present the known results on equidistribution
of lattice points on spheres (as in Definition \ref{def:admissible-sequence})
in various dimensions $d$.

In dimension $d\ge5$, it was shown in \cite{Pommerenke} that we
have equidistribution unconditionally, and any sequence $L\to\infty$
(with $L^{2}\in\z$) is admissible.

When $d=4$, any natural number is a sum of four squares, but there
are arbitrarily large values of $L$ with few representations, and
$\dim\hilbert$ may remain bounded as $L\to\infty$. Requiring $\dim\hilbert\to\infty$
(for instance, by bounding the multiplicity of the prime $2$ in $L^{2}$)
yields equidistribution, again by \cite{Pommerenke} (see also \cite{Malyshev-four-dimensional-equidistribution}).

In dimension $d=3$, a congruence condition $L^{2}\not\equiv0,4,7\pmod8$
ensures $\dim\hilbert\to\infty$ (thus, bounding the multiplicity
of the prime $2$ in $L^{2}$ also ensures this). The question of
whether $\dim\hilbert\to\infty$ implies equidistribution is very
difficult, and was answered affirmatively in \cite{Golubeva-Fomenko}
and \cite{Duke} following a breakthrough by Iwaniec \cite{Iwaniec-Fourier-coefficients-of-modular-forms-of-half-integral-weight}.
See also \cite{Duke-Introduction}.

The equidistribution question in dimension $d=2$ is trickier than
higher dimensions. Any condition that simply ensures $\dim\hilbert\to\infty$
must strongly depend on the prime decomposition of $L^{2}$, as can
be concluded from Gauss's classical formula for the number of representations
of integers as sums of two squares. Furthermore, it turns out that
the condition $\dim\hilbert\to\infty$ is not strong enough to ensure
equidistribution, and the limit measure may even be a sum of 4 atoms,
as shown in \cite{Cilleruelo-1993}. On the positive side, equidistribution
can be proven for a subsequence of relative density $1$ in the sequence
of sums of two squares, as shown in \cite{Katai-Kornyei} and \cite{Erdos-Hall}
(see also \cite{Fainsilber-Kurlberg-Wennberg}).

In case that $d=2$ and the integer points on the circle accumulate
according to a non-uniform limiting measure, if it has no atoms, then
the result of Theorem \ref{thm:asymptotic-law} still holds. However,
the value of $\nu$ depends on the limiting measure. This situation
is further investigated in \cite{Kurlberg-Wigman-Non-universality};
see also \cite{Buckley-Wigman-On-the-number-of-nodal-domains-of-toral-eigenfunctions}.

Regarding the number of lattice points, when $d\ge4$ we have $\dim\hilbert\gg L^{d-2}$
(under the correct assumptions when $d=4$) by the classical Hardy-Littlewood
circle method - see, for instance, \cite[Chapter 12]{Grosswald-book-Representations-of-integers-as-sums-of-squares}.
When $d=3$, we have (under the correct assumptions) $\dim\hilbert\gg c\left(\delta\right)L^{1-\delta}$
for any fixed $0<\delta<1$, due to Siegel - see \cite[Chapter 21]{Davenport-book-Multiplicative-number-theory}.
Finally, when $d=2$, we may find equidistributed subsequences of
relative density $1$ that satisfy $\dim\hilbert\gg\left(\log L\right)^{\gamma}$
for any fixed $0<\gamma<\frac{1}{2}\log\frac{\pi}{2}\approx0.226$,
by \cite{Erdos-Hall}.

\section{\label{sec:appendix-additional-proofs}Additional proofs}

\subsection{Stability of nodal sets - the ``shell lemma'' and Propositions
\ref{prop:perturbation-in-U} and \ref{prop:smooth-perturbation}}

Let $\alpha,\beta>0$ and let $U$ be an open subset of $\td$. Let
$f\colon U\to\R$ be a smooth function such that $\left|f\left(x\right)\right|>\alpha$
or $\left|\nabla f\left(x\right)\right|>\beta$ for any $x\in U$.

The ``shell lemma'', given below (cf.\ \cite[Claim 4.2]{Nazarov-Sodin-spherical-harmonics}),
shows that each connected component of $\left\{ x\in U:f\left(x\right)=0\right\} $
which is not too close to $\partial U$ is contained in a ``shell'',
which is a connected component of $\left\{ x\in U:\left|f\left(x\right)\right|<\alpha\right\} $,
and the shells satisfy certain properties. Proposition \ref{prop:perturbation-in-U}
follows immediately from this lemma, and the proof of Proposition
\ref{prop:smooth-perturbation}, given below, also follows from it.

Before presenting the lemma and its proof, we construct a vector field
whose integral curves are used in the proof. See, for instance, \cite[Chapter 9]{Lee-book-Introduction-to-Smooth-Manifolds-second-ed}
for the necessary background in the theory of integral curves and
flows on smooth manifolds.

Let $M\coloneqq\left\{ x\in U:\left|\nabla f\left(x\right)\right|>\beta\right\} $.
On the open submanifold $M$, define the following vector field:
\[
V\coloneqq\frac{\nabla f}{\left|\nabla f\right|^{2}}.
\]
For any $p\in M$, let $\theta^{\left(p\right)}\colon\D^{\left(p\right)}\to M$
be the integral curve starting at $p$ with respect to the vector
field $V$ (where $\D^{\left(p\right)}$ is an open interval containing
zero, the curve's maximal domain). It is easy to see that this integral
curve has the following three properties (see Figure \ref{fig:integral-curve}):
\begin{enumerate}
\item $f\left(\theta^{\left(p\right)}\left(t\right)\right)=f\left(p\right)+t$
for any $t\in\D^{\left(p\right)}$ (because the left side has constant
derivative $1$).
\item $\dist\left(p,\theta^{\left(p\right)}\left(t\right)\right)\le t/\beta$
(because $\left|V\right|\le\beta^{-1}$).
\item If $p$ is such that $f\left(p\right)=0$ and $\overline{B}\left(p,\alpha/\beta\right)\subset U$,
then $\left[-\alpha,\alpha\right]\subset\D^{\left(p\right)}$.
\end{enumerate}
\noindent 
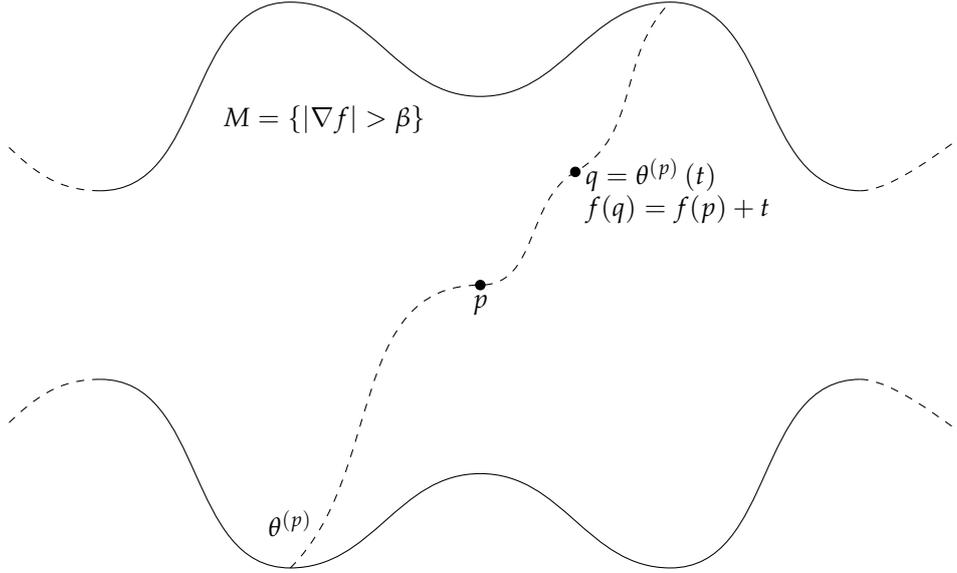
\begin{figure}[h]
\centering{}\begin{tikzpicture}[scale=2.5]
\coordinate (outer1) at (0,0);
\coordinate (outer2) at (1, 1.5);
\coordinate (outer3) at (2, 1);
\coordinate (outer4) at (3, 1.5);
\coordinate (outer5) at (4, 0);
\coordinate (outer6) at (3, -1.5);
\coordinate (outer7) at (2, -1);
\coordinate (outer8) at (1, -1.5);

\coordinate (top0) at (-0.5, 0.75);
\coordinate (top1) at (0,0.5);
\coordinate (top2) at (1, 1.5);
\coordinate (top3) at (2, 1);
\coordinate (top4) at (3, 1.5);
\coordinate (top5) at (4, 0.5);
\coordinate (top6) at (4.5, 0.75);

\coordinate (bot0) at (-0.5, -0.75);
\coordinate (bot1) at (0,-0.5);
\coordinate (bot2) at (1, -1.5);
\coordinate (bot3) at (2, -1);
\coordinate (bot4) at (3, -1.5);
\coordinate (bot5) at (4, -0.5);
\coordinate (bot6) at (4.5, -0.75);

\coordinate (P) at (2, 0);
\coordinate (Q) at (2.5, 0.6);
\coordinate (path1) at (outer8);
\coordinate (path2) at (P);
\coordinate (path3) at (Q);
\coordinate (path4) at (outer4);

\draw (top1) to [out=0, in=180] (top2) to [out=0, in=180] (top3) to [out=0, in=180] (top4) to [out=0, in=180] (top5);
\draw [dashed] (top1) to [out=180, in=-45] (top0);
\draw [dashed] (top5) to [out=0, in=45] (top6);

\draw (bot1) to [out=0, in=180] (bot2) to [out=0, in=180] (bot3) to [out=0, in=180] (bot4) to [out=0, in=180] (bot5);
\draw [dashed] (bot1) to [out=180, in=45] (bot0);
\draw [dashed] (bot5) to [out=0, in=-45] (bot6);

\node [below left, xshift=-6mm] at (top3) {$M=\left\{\left|\nabla f\right|>\beta\right\}$};

\draw [dashed] (path1) to [out=45, in=180] (path2) to [out=0, in=180+30] (path3) to [out=30, in=180+45] (path4);

\draw [fill] (P) circle [radius=0.025];
\node [below] at (P) {$p$};

\draw [fill] (Q) circle [radius=0.025];
\node [anchor=north west,yshift=3mm,xshift=-2mm] at (Q) {\begin{tabular}{l} $q=\theta^{\left( p \right)}\left( t \right)$ \\ $f(q)=f(p)+t$ \end{tabular}};

\node [anchor=south,yshift=3mm] at (path1) {$\theta^{\left( p \right)}$};

\end{tikzpicture}\caption{\label{fig:integral-curve}Starting at a point $p\in M$, the parameterization
of the integral curve $\theta^{\left(p\right)}$ corresponds to the
change in value of $f$. The distance between $p$ and $q=\theta^{\left(p\right)}\left(t\right)$
is at most $t/\beta$.}
\end{figure}

Following the above construction, we now formulate and prove the shell
lemma. The shells described by the lemma are illustrated in Figure
\ref{fig:shell-lemma}.
\begin{lem}[the shell lemma]
\label{lem:shell}$ $
\selectlanguage{british}%
\begin{enumerate}[labelindent=\parindent,leftmargin=2em,label=({\roman*})]
\item \foreignlanguage{american}{\label{enu:shell-prop-1}Each connected
component $\Gamma$ of $\left\{ x\in U:f\left(x\right)=0\right\} $
that satisfies $\Gamma_{+\alpha/\beta}\subset U$ is contained in
an open, connected ``shell'' $S_{\Gamma}\subset\left\{ x\in U:\left|f\left(x\right)\right|<\alpha\right\} $
whose boundary consists of two components, with $f=\alpha$ on one
and $f=-\alpha$ on the other.}
\selectlanguage{american}%
\item \label{enu:shell-prop-2}$S_{\Gamma}\subset\Gamma_{+\alpha/\beta}$,
and for any point $p\in\Gamma$, the ball $\overline{B}\left(p,\alpha/\beta\right)$
contains a path through $p$ from one boundary component of $S_{\Gamma}$
to the other.
\item \label{enu:shell-prop-3}Given two such components $\Gamma_{1}\ne\Gamma_{2}$,
the shells $S_{\Gamma_{1}},S_{\Gamma_{2}}$ are disjoint.
\item \label{enu:shell-prop-4}$S_{\Gamma}$ may be decomposed as $\Gamma\cup S_{\Gamma}^{+}\cup S_{\Gamma}^{-}$,
where $S_{\Gamma}^{+}=\left\{ x\in S_{\Gamma}:f\left(x\right)>0\right\} $,
$S_{\Gamma}^{-}=\left\{ x\in S_{\Gamma}:f\left(x\right)<0\right\} $.
In this decomposition, $S_{\Gamma}^{+}$ and $S_{\Gamma}^{-}$ are
connected open sets.
\item \label{enu:shell-prop-5}In case $U=\td$, the shells $\left\{ S_{\Gamma}\right\} _{\Gamma\in\Z\left(f\right)}$
are precisely the connected components of $\left\{ x\in\td:\left|f\left(x\right)\right|<\alpha\right\} $.
\end{enumerate}
\end{lem}
\noindent 
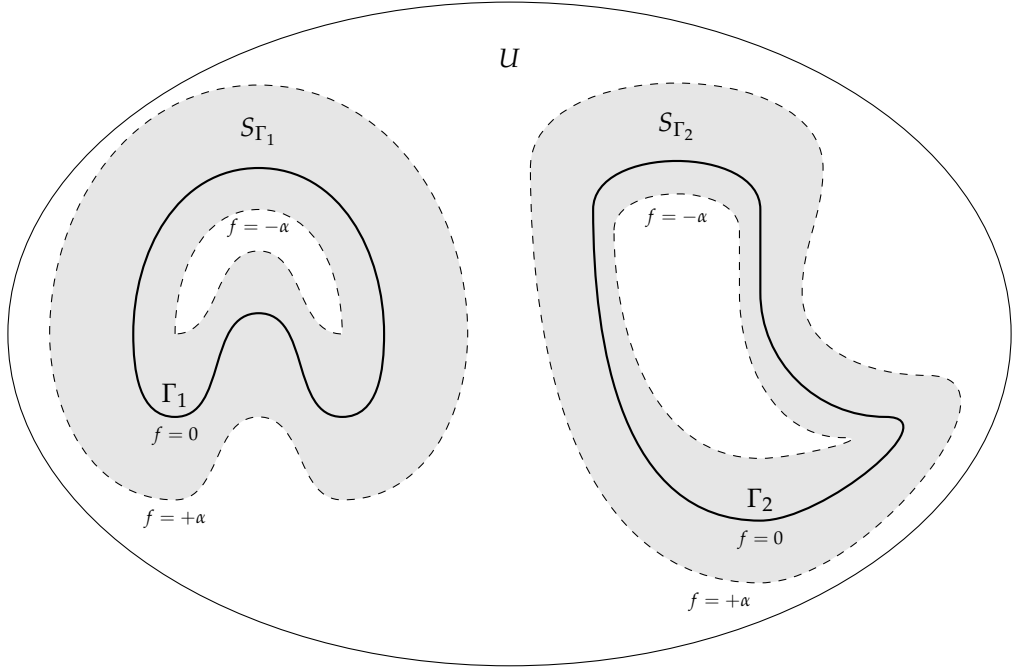
\begin{figure}[h]
\begin{centering}
\definecolor{fillgray}{gray}{0.9}
\begin{tikzpicture}[scale=2.2]

\coordinate (origin) at (0,0);
\coordinate (circle_top) at (0,2);

\coordinate (a1) at (-2.25, 0);
\coordinate (a2) at (-1.5,1);
\coordinate (a3) at (-0.75,0);
\coordinate (a4) at (-1,-0.5);
\coordinate (a5) at (-1.5,0.125);
\coordinate (a6) at (-2,-0.5);

\coordinate (ao1) at (-2.75, 0);
\coordinate (ao2) at (-1.5,1.5);
\coordinate (ao3) at (-0.25,0);
\coordinate (ao4) at (-1,-1);
\coordinate (ao5) at (-1.5,-0.5);
\coordinate (ao6) at (-2,-1);

\coordinate (ai1) at (-2, 0);
\coordinate (ai2) at (-1.5,0.75);
\coordinate (ai3) at (-1,0);
\coordinate (ai4) at (-1,0);
\coordinate (ai5) at (-1.5,0.5);
\coordinate (ai6) at (-2,0);

\coordinate (b1) at (0.5,0.75);
\coordinate (b2) at (1.5,0.75);
\coordinate (b3) at (1.5,0.25);
\coordinate (b4) at (2.25,-0.5);
\coordinate (b5) at (1.5,-1.125);

\coordinate (bo1) at (0.125,1);
\coordinate (bo2) at (1.875,1);
\coordinate (bo1_bo2_middle) at (1,1.25); 
\coordinate (bo3) at (1.75,0.25);
\coordinate (bo4) at (2.5,-0.25);
\coordinate (bo5) at (1.5,-1.5);

\coordinate (bi1) at (0.625,0.625);
\coordinate (bi2) at (1.375,0.625);
\coordinate (bi1_bi2_middle) at (1,0.7); 
\coordinate (bi3) at (1.375,0.25);
\coordinate (bi4) at (2,-0.625);
\coordinate (bi5) at (1.5,-0.75);

\draw (origin) circle [x radius=3, y radius = 2];
\node [below, yshift=-5mm] at (circle_top) {$U$};

\draw [dashed, fill=fillgray] (ao1) to [out=90, in=180] (ao2) to [out=0, in=90] (ao3) to [out=270, in=0] (ao4) to [out=180, in=0] (ao5) to [out=180, in=0] (ao6) to [out=180, in=270] (ao1);
\draw [dashed, fill=white] (ai1) to [out=90, in=180] (ai2) to [out=0, in=90] (ai3) to [out=270, in=0] (ai4) to [out=180, in=0] (ai5) to [out=180, in=0] (ai6) to [out=180, in=270] (ai1);
\draw [thick] (a1) to [out=90, in=180] (a2) to [out=0, in=90] (a3) to [out=270, in=0] (a4) to [out=180, in=0] (a5) to [out=180, in=0] (a6) to [out=180, in=270] (a1);

\node [above] at (a6) {$\Gamma_1$};
\node [below, yshift=-3mm] at (ao2) {$S_{\Gamma_1}$};
\node [below] at (a6) {\scriptsize $f=0$};
\node [below] at (ao6) {\scriptsize $f=+\alpha$};
\node [below] at (ai2) {\scriptsize $f=-\alpha$};

\draw [dashed, fill=fillgray] (bo1) to [out=90, in=90] (bo2) to [out=270, in=90] (bo3) to [out=270, in=180] (bo4) to [out=0, in=0] (bo5) to [out=180, in=270] (bo1);
\draw [dashed, fill=white] (bi1) to [out=90, in=90] (bi2) to [out=270, in=90] (bi3) to [out=270, in=180] (bi4) to [out=0, in=0] (bi5) to [out=180, in=270] (bi1);
\draw [thick] (b1) to [out=90, in=90] (b2) to [out=270, in=90] (b3) to [out=270, in=180] (b4) to [out=0, in=0] (b5) to [out=180, in=270] (b1);

\node [above] at (b5) {$\Gamma_2$};
\node at (bo1_bo2_middle) {$S_{\Gamma_2}$};
\node [below] at (b5) {\scriptsize $f=0$};
\node [below left] at (bo5) {\scriptsize $f=+\alpha$};
\node at (bi1_bi2_middle) {\scriptsize $f=-\alpha$};

\end{tikzpicture}
\par\end{centering}

\centering{}\caption{\label{fig:shell-lemma}Illustration of two connected components of
$\left\{ x\in U:f\left(x\right)=0\right\} $ and their corresponding
shells.}
\end{figure}

\begin{proof}
Let $\Gamma$ be a connected component of $\left\{ x\in U:f\left(x\right)=0\right\} $
that satisfies $\Gamma_{+\alpha/\beta}\subset U$. Define:
\[
S_{\Gamma}\coloneqq\left\{ \theta^{\left(x_{0}\right)}\left(t\right):\left|t\right|<\alpha,x_{0}\in\Gamma\right\} .
\]

For any $p\in S_{\Gamma}$, if $p=\theta^{\left(x_{0}\right)}\left(t\right)$,
then $t=f\left(p\right)$ and $x_{0}=\theta^{\left(p\right)}\left(-t\right)$.
Therefore, the smooth map $\left(-\alpha,\alpha\right)\times\Gamma\to S_{\Gamma}$
given by $\left(t,x_{0}\right)\mapsto\theta^{\left(x_{0}\right)}\left(t\right)$
has a smooth inverse, and we have that $\left(-\alpha,\alpha\right)\times\Gamma$
is diffeomorphic to $S_{\Gamma}$. Therefore $S_{\Gamma}$ is open,
connected, and has exactly two boundary components, both diffeomorphic
to $\Gamma$, with $f=\alpha$ on one and $f=-\alpha$ on the other,
proving \ref{enu:shell-prop-1}. The path described in \ref{enu:shell-prop-2}
is given by $\theta^{\left(p\right)}\left(\left[-\alpha,\alpha\right]\right)$.
For \ref{enu:shell-prop-3}, note that for each $p\in S_{\Gamma}$
we have $\theta^{\left(p\right)}\left(-f\left(p\right)\right)\in\Gamma$,
so $p$ cannot be simultaneously in $S_{\Gamma_{1}}$ and $S_{\Gamma_{2}}$.
\ref{enu:shell-prop-4} follows from the fact that $S_{\Gamma}^{+}$
and $S_{\Gamma}^{-}$ are continuous images of the connected sets
$\left(0,\alpha\right)\times\Gamma$ and $\left(-\alpha,0\right)\times\Gamma$,
respectively.

Finally, we prove \ref{enu:shell-prop-5}. Suppose $U=\td$; the shells
$S_{\Gamma}$ are among the connected components of $\left\{ x\in\td:\left|f\left(x\right)\right|<\alpha\right\} $,
and it remains to show that there is no other connected component
$S$. We have $f=\pm\alpha$ on $\partial S$, but $f$ cannot be
constant on $\partial S$ or there would have to be a point inside
$S$ where $\nabla f=0$, a contradiction. Therefore there must be
some $\Gamma\in\Z\left(f\right)$ lying inside $S$, and we have $S=S_{\Gamma}$.
\end{proof}

\begin{proof}[Proof of Proposition \ref{prop:smooth-perturbation}]
By applying Proposition \ref{prop:perturbation-in-U} directly, we
get $N\left(f\right)\le N\left(f+g\right)$, and by applying Proposition
\ref{prop:perturbation-in-U} on the function $f+g$ with perturbation
$-g$, we get $N\left(f+g\right)\le N\left(f\right)$. Therefore,
the mapping $\Gamma\mapsto\tilde{\Gamma}$ given in Proposition \ref{prop:perturbation-in-U}
is a bijection, and each $\tilde{\Gamma}$ is the \emph{only} component
of $\Z\left(f+g\right)$ lying inside $S_{\Gamma}$. Let $p,q\in\Gamma$
be a pair of points realizing the diameter: $\diam\Gamma=\dist\left(p,q\right)$.
By property \ref{enu:shell-prop-2} in Lemma \ref{lem:shell}, the
ball $\overline{B}\left(p,\alpha/\beta\right)$ contains a point $\tilde p\in Z\left(f+g\right)$,
which must then belong to $\tilde{\Gamma}$. Similarly, the ball $\overline{B}\left(q,\alpha/\beta\right)$
contains a point $\tilde q\in\tilde{\Gamma}$. By the triangle inequality,
\[
\diam\Gamma=\dist\left(p,q\right)\le\dist\left(p,\tilde p\right)+\dist\left(\tilde p,\tilde q\right)+\dist\left(\tilde q,q\right)\le\frac{\alpha}{\beta}+\diam\tilde{\Gamma}+\frac{\alpha}{\beta}.\qedhere
\]

\end{proof}

\subsection{Counting nodal components vs.\ counting nodal domains - Proposition
\ref{prop:nodal-components-similar-to-nodal-domains}}

We prove Proposition \ref{prop:nodal-components-similar-to-nodal-domains}
using an elementary concept in singular homology theory - the Mayer-Vietoris
sequence (see, for instance, \cite[Section 2.2]{Hatcher-book-Algebraic-Topology}).
We denote by $H_{n}\left(X\right)$ the $n$th singular homology group
of the topological space $X$, and by $\cong$ an isomorphism of groups.
\begin{proof}[Proof of Proposition \ref{prop:nodal-components-similar-to-nodal-domains}]
Let $\alpha,\beta>0$ be such that $\left|f\left(x\right)\right|>\alpha$
or $\left|\nabla f\left(x\right)\right|>\beta$ for any $x\in\td$
(see Remark \ref{rem:about-nodal-stability}). Define two sets $A,B\subset\td$
by:
\begin{align*}
A & \coloneqq\left\{ x\in\td:f\left(x\right)\ne0\right\} \\
B & \coloneqq\left\{ x\in\td:\left|f\left(x\right)\right|<\alpha\right\} 
\end{align*}
$A$ and $B$ are both open, and $A\cup B=\td$. We count the connected
components of $A,B$ and $A\cap B$:
\begin{itemize}
\item The connected components of $A$ are precisely the nodal domains,
so $A$ has $r$ components.
\item The connected components of $B$ are precisely the shells $S_{\Gamma}$
defined in Lemma \ref{lem:shell}, which are in correspondence with
the nodal components, so $B$ has $k$ components.
\item By Lemma \ref{lem:shell} \ref{enu:shell-prop-4}, excluding $\Gamma$
from any shell $S_{\Gamma}$ leaves it with exactly two components,
so $A\cap B$ has $2k$ components.
\end{itemize}
Consider the last four terms of the Mayer-Vietoris sequence, and name
the nonzero maps $\phi_{1},\phi_{2},\phi_{3}$:
\[
\cdots\longrightarrow H_{1}\left(\td\right)\overset{\phi_{1}}{\longrightarrow}H_{0}\left(A\cap B\right)\overset{\phi_{2}}{\longrightarrow}H_{0}\left(A\right)\oplus H_{0}\left(B\right)\overset{\phi_{3}}{\longrightarrow}H_{0}\left(\td\right)\longrightarrow0.
\]
Written explicitly:
\[
\cdots\longrightarrow\zd\overset{\phi_{1}}{\longrightarrow}\z^{2k}\overset{\phi_{2}}{\longrightarrow}\z^{r+k}\overset{\phi_{3}}{\longrightarrow}\z\longrightarrow0.
\]
$\phi_{1},\phi_{2},\phi_{3}$ are group homomorphisms, and they may
be extended naturally to linear maps between $\q$-vector spaces,
so the rank-nullity theorem applies and we get:
\begin{equation}
\begin{aligned}d & =\rank\img\phi_{1}+\rank\ker\phi_{1}\\
2k & =\rank\img\phi_{2}+\rank\ker\phi_{2}\\
r+k & =\rank\img\phi_{3}+\rank\ker\phi_{3}
\end{aligned}
\label{eq:rank-nullity}
\end{equation}
By the exactness of the Mayer-Vietoris sequence:
\begin{equation}
\begin{aligned}\img\phi_{1} & =\ker\phi_{2}\\
\img\phi_{2} & =\ker\phi_{3}\\
\img\phi_{3} & =\z
\end{aligned}
\label{eq:by-exactness}
\end{equation}
Plugging equations (\ref{eq:by-exactness}) into equations (\ref{eq:rank-nullity}),
we get:
\[
\rank\ker\phi_{1}=d-k+r-1.
\]
Since $\ker\phi_{1}$ is a subgroup of $\zd$, we have $0\le\rank\ker\phi_{1}\le d$,
leading to the required conclusion.

For the second part, let $\Gamma_{1},\ldots,\Gamma_{k'}\in\Z\left(f\right)$
be the nodal components of $f$ that lie completely inside some open
ball $U$ with radius less than $\frac{1}{2}$. Suppose the set $A\coloneqq U\setminus\left(\Gamma_{1}\cup\dots\cup\Gamma_{k'}\right)$
has $s$ components. Exactly one touches the (connected) boundary
of $U$. The other $s-1$ have $f=0$ on their boundary, and each
of them must contain at least one nodal domain of $f$ that lies completely
inside $U$, so $s-1\le r'$.

We conclude by showing that $s=k'+1$. Let $\alpha,\beta>0$ be such
that $\left|f\left(x\right)\right|>\alpha$ or $\left|\nabla f\left(x\right)\right|>\beta$
for any $x\in\td$, and $\dist\left(\Gamma_{i},\partial U\right)\ge\alpha/\beta$
for any $1\le i\le k'$ (see Remark \ref{rem:about-nodal-stability}).
By Lemma \ref{lem:shell}, each $\Gamma_{i}$ is contained in a shell
$S_{\Gamma_{i}}\subset U$. Define $B\coloneqq S_{\Gamma_{1}}\cup\dots\cup S_{\Gamma_{k'}}$.
As in the first part, $A,B$ are open and $A\cup B=U$, $A$ has $s$
components, $B$ has $k'$ components, and $A\cap B$ has $2k'$ components.
$U$ is simply-connected (because its radius is less than $\frac{1}{2}$),
so the last four terms of the Mayer-Vietoris sequence form the short
exact sequence $0\longrightarrow\z^{2k'}\longrightarrow\z^{s+k'}\longrightarrow\z\longrightarrow0$,
and we get $s=k'+1$.
\end{proof}

\subsection{Local bounds on eigenfunctions - Proposition \ref{prop:local-bounds}}
\begin{proof}[Proof of Proposition \ref{prop:local-bounds}]
W.l.o.g.\ we assume that $x_{0}=0$, and also that $L=1$ - for
the latter, simply apply the result to $\tilde f\left(x\right)\coloneqq f\left(x/L\right)$.
Thus, $f$ satisfies $\Delta f+4\pi^{2}f=0$.

We utilize a trick to view $f$ as a harmonic function in one more
dimension. Define $u\colon\R^{d+1}\to\R$ by:
\[
u\left(x_{1},\ldots,x_{d},x_{d+1}\right)\coloneqq f\left(x_{1},\ldots,x_{d}\right)\cosh\left(2\pi x_{d+1}\right).
\]
Clearly, $u$ is harmonic and $u\left(0\right)=f\left(0\right)$,
$\frac{\partial u}{\partial x_{j}}\left(0\right)=\frac{\partial f}{\partial x_{j}}\left(0\right)$
and $\frac{\partial^{2}u}{\partial x_{i}\partial x_{j}}\left(0\right)=\frac{\partial^{2}f}{\partial x_{i}\partial x_{j}}\left(0\right)$
for $1\le i,j\le d$. Using the mean value property in the ball $B^{d+1}\left(0,r/\sqrt{d}\right)$,
we get:
\[
f\left(0\right)=C\underset{B^{d+1}\left(0,r/\sqrt{d}\right)}{\idotsint}f\left(x_{1},\ldots,x_{d}\right)\cosh\left(2\pi x_{d+1}\right)\,\mathrm{d}x_{1}\cdots\mathrm{d}x_{d+1},
\]
where $C$ is a positive constant depending only on $r,d$. Taking
the absolute value and using the inclusions
\[
B^{d+1}\left(0,\frac{r}{\sqrt{d}}\right)\subset\left[-\frac{r}{\sqrt{d}},\frac{r}{\sqrt{d}}\right]^{d+1}\qquad\text{and}\qquad\left[-\frac{r}{\sqrt{d}},\frac{r}{\sqrt{d}}\right]^{d}\subset B^{d}\left(0,r\right),
\]
we get:
\begin{align*}
\left|f\left(0\right)\right| & \ll\underset{\left[-r/\sqrt{d},r/\sqrt{d}\right]^{d+1}}{\idotsint}\left|f\left(x_{1},\ldots,x_{d}\right)\right|\left|\cosh\left(2\pi x_{d+1}\right)\right|\,\mathrm{d}x_{1}\cdots\mathrm{d}x_{d+1}\ll\Int[B^{d}\left(0,r\right)]{\left|f\left(x\right)\right|}x,
\end{align*}
where the implied constant in $\ll$ depends only on $r,d$. By the
Cauchy-Schwarz inequality, we get (\ref{eq:bound-function}).

Next, we turn to (\ref{eq:bound-gradient}) and (\ref{eq:bound-hessian}).
Note that for any $\left|x\right|<\frac{r}{2}$ in $\R^{d+1}$, we
have:
\[
u\left(x\right)=\left(\frac{r}{2}\right)^{d-1}\Int[\left|\zeta\right|=1]{u\left(\frac{1}{2}r\zeta\right)P\left(x,\frac{1}{2}r\zeta\right)}{\sigma_{d}\left(\zeta\right)},
\]
where ${\displaystyle P\left(x,y\right)\coloneqq\frac{\left|x\right|^{2}-\left|y\right|^{2}}{\left|x-y\right|^{d+1}}}$
is the $\left(d+1\right)$-dimensional Poisson kernel. Differentiating
under the integral sign with respect to $x_{j}$, we get:
\[
\frac{\partial f}{\partial x_{j}}\left(0\right)=\frac{\partial u}{\partial x_{j}}\left(0\right)=\left(\frac{r}{2}\right)^{d-1}\Int[\left|\zeta\right|=1]{u\left(\frac{1}{2}r\zeta\right)\frac{\partial P}{\partial x_{j}}\left(0,\frac{1}{2}r\zeta\right)}{\sigma_{d}\left(\zeta\right)}.
\]

The Poisson kernel is smooth when its two parameters are separated,
so $\zeta\mapsto\frac{\partial P}{\partial x_{j}}\left(0,\frac{1}{2}r\zeta\right)$
is a continuous function on $\left\{ \left|\zeta\right|=1\right\} $.
Its integral depends only on $r,d$. Therefore:
\[
\left|\frac{\partial f}{\partial x_{j}}\left(0\right)\right|\ll\sup_{\left|\zeta\right|=1}\left|u\left(\frac{1}{2}r\zeta\right)\right|\le\sup_{\left|x\right|\le r/2}\left|f\left(x\right)\right|\sup_{\left|t\right|\le r/2}\left|\cosh\left(2\pi t\right)\right|\ll\sup_{\left|x\right|\le r/2}\left|f\left(x\right)\right|.
\]
Using (\ref{eq:bound-function}) applied on $f$ at point $x$, this
gives:
\[
\left|\nabla f\left(0\right)\right|^{2}=\sum_{j=1}^{d}\left|\frac{\partial f}{\partial x_{j}}\left(0\right)\right|^{2}\ll\sup_{\left|x\right|\le r/2}\left|f\left(x\right)\right|^{2}\ll\sup_{\left|x\right|\le r/2}\Int[B^{d}\left(x,r/2\right)]{\left|f\left(y\right)\right|^{2}}y\le\Int[B^{d}\left(0,r\right)]{\left|f\left(y\right)\right|^{2}}y.
\]
This proves (\ref{eq:bound-gradient}). Similarly, applying (\ref{eq:bound-gradient})
on $\frac{\partial f}{\partial x_{j}}$ at point $0$ and then applying
(\ref{eq:bound-function}) on $\frac{\partial f}{\partial x_{j}}$
at point $x$, we get:
\[
\left|\nabla\frac{\partial f}{\partial x_{j}}\left(0\right)\right|^{2}\ll\Int[B^{d}\left(0,r/2\right)]{\left|\frac{\partial f}{\partial x_{j}}\left(x\right)\right|^{2}}x\ll\sup_{\left|x\right|\le r/2}\left|\frac{\partial f}{\partial x_{j}}\left(x\right)\right|^{2}\ll\Int[B^{d}\left(0,r\right)]{\left|f\left(y\right)\right|^{2}}y.
\]
Summing over $j$ gives (\ref{eq:bound-hessian}).
\end{proof}
\bibliographystyle{amsalpha-initials}
\phantomsection\addcontentsline{toc}{section}{\refname}\bibliography{Article}

\end{document}